\documentclass[10pt,a4paper,reqno]{amsart}

\usepackage{lmodern}
\usepackage[T1]{fontenc}
\usepackage[utf8]{inputenc}
\usepackage[english]{babel}
\usepackage{microtype} 
\usepackage[yyyymmdd,hhmmss]{datetime}

\usepackage{slashed,url,bm,amssymb,mathrsfs,mathtools,todonotes,xparse,booktabs,graphicx}
\usepackage[centertableaux]{ytableau}
\makeatletter
\edef\boxframe@normal@YT{\boxframe@YT}
\usepackage[pdftex]{hyperref}
\usepackage{float,placeins}

\usepackage{blindtext} 
\usepackage{epigraph} 

\usepackage{array}

\usepackage{fancyhdr}
\usepackage[font={small,sl}]{caption}
\newtheorem{theorem}{Theorem}[section]
\newtheorem{lemma}[theorem]{Lemma}
\newtheorem{definition}[theorem]{Definition}

\newtheorem{corollary}[theorem]{Corollary}
\newtheorem{conjecture}[theorem]{Conjecture}
\newtheorem{example}[theorem]{Example}

\definecolor{pBlue}{RGB}{86,139,190}
\definecolor{pCyan}{RGB}{149,186,201}
\definecolor{pSand}{RGB}{184,166,121}
\definecolor{pAlgae}{RGB}{87,115,135}
\definecolor{pSkin}{RGB}{236,216,167}
\definecolor{pGray}{RGB}{156,175,156}
\definecolor{pPink}{RGB}{215,114,127}
\definecolor{pOrange}{RGB}{211,153,80}
\definecolor{pDarkGreen}{RGB}{43,140,51}

\definecolor{boxL}{RGB}{153,211,80}
\definecolor{boxU}{RGB}{80,153,211}
\definecolor{boxS}{RGB}{236,216,167}
\definecolor{boxG}{RGB}{240,240,240}
\definecolor{boxB}{RGB}{0,0,0}
\definecolor{boxGr}{RGB}{205,205,205}

\DeclareMathOperator{\Res}{Res}
\DeclareMathOperator{\arm}{arm}
\DeclareMathOperator{\leg}{leg}

\newcommand{\setZ}{\mathbb{Z}}
\newcommand{\setR}{\mathbb{R}}

\newcommand{\jackj}{J}

\newcommand{\jackjLR}{g}
\newcommand{\JackjLR}{{\bm{g}}}

\newcommand{\bsigma}{\bar{\sigma}}
\newcommand{\bmu}{\bar{\mu}}

\newcommand{\bmD}{{\bm {D}}}

\newcommand{\bmss}{{\bm {s}}}
\newcommand{\mk}{{\cdot}}

\newcommand{\norm}[1]{\lVert#1\rVert} 

\newcommand{\uhc}[2]{{\bm{h}^{\scriptscriptstyle{\mathrm U}}_{#1}(#2)}}
\newcommand{\lhc}[2]{{\bm{h}^{\scriptscriptstyle{\mathrm L}}_{#1}(#2)}}

\newcommand{\ula}[2]{{{\bm{h}}^\bullet_{#1}(#2) }}

\newcommand{\bk}[0]{{\bm{k}}}

\newcommand{\bmk}{{\bm{k}}}

\newcommand{\sU}{{\scriptscriptstyle{\mathrm U}}}
\newcommand{\sL}{{\scriptscriptstyle{\mathrm L}}}
\newcommand{\sA}{{\scriptscriptstyle{\mathrm A}}}
\newcommand{\sB}{{\scriptscriptstyle{\mathrm B}}}

\newcommand{\uh}[2]{{h^{\scriptscriptstyle{\mathrm U}}_{#1}(#2)}}
\newcommand{\ulh}[2]{{h^{\scriptscriptstyle{\mathrm U/L}}_{#1}(#2)}}
\newcommand{\lh}[2]{{h^{\scriptscriptstyle{\mathrm L}}_{#1}(#2)}}

\newcommand{\bmin}[2]{{#1 \mathbin{\curlywedge} #2}}
\newcommand{\bmax}[2]{{#1 \mathbin{\curlyvee} #2}}

\DeclareMathSymbol{\shortminus}{\mathbin}{AMSa}{"39}

\newcommand{\BZ}{\mathbb{Z}}

\newcommand{\StD}{\mathsf{SD}}

\newcommand{\StS}{\mathsf{SS}}

\newcommand{\windfact}{\bm{F}}

\newcommand{\hkl}[2]{{{\lfloor}\!{#1,\!#2} \!{\rfloor}\!}}

\newcommand{\meetjoin}{{\Diamond}}

\title[The Stanley conjecture revisited]{The Stanley conjecture revisited 
     }

\author[R. Mickler]{Ryan Mickler}
\address{Singulariti Research, 55 University Street, Carlton, 3053, Australia}
\email{ry.mickler@gmail.com}

\begin{document}

\begin{abstract}
In the seminal work of Stanley \cite{Stanley:1989}, several conjectures were made on the structure of Littlewood-Richardson coefficients for the multiplication of Jack symmetric functions. Motivated by recent results of Alexandersson and the present author \cite{Alexandersson:2023}, we postulate that a `windowing' property holds for all such Jack L-R coefficients. Furthermore, we propose an extension of the `Factorization' property for Schur L-R due to King-Tollu-Toumazet {\cite{King:2009}} to the Jack case. These properties provide a vast set of relations between the Jack L-R coefficients and allow for their direct computation in a certain large class of cases.  \\
\end{abstract}

\maketitle


\section*{Structure of paper}

In Section \ref{section:prelim}, we begin by providing a review the titular Stanley Conjectures of \cite{Stanley:1989}. Our first novel conjecture \eqref{conj:generalstructure} on the general structure of Jack Littlewood-Richardson (LR) coefficients provides a modest but straightforward generalization of the strong form of Stanley's conjecture \eqref{conj:strongstanley}. We then provide a summary of some relevant known results, and revisit the recently proven case of a rectangular union \cite{Alexandersson:2023}, which was the primary motivation for this present work.

In Section \ref{section:windowing}, we deliver a wide ranging generalization of the Rectangular Union result in one of the central conjectures of this paper: the Jack Windowing Conjecture \eqref{conj:jackwindowing}. We show that this conjecture holds in several known cases of the Stanley Conjecture (where the corresponding Schur LR coefficient is $c=1$), including the Maximal Filling Tableau case (which was proven by Stanley). 
The work of Bravi-Gandini \cite{Bravi:2021} has substantial overlap with the topics of this section, and we use their work to prove a significant case of the Jack Windowing conjecture. We provide further evidence for this conjecture by showing that it produces a specific solution to a 6-parameter family of $c=2$ examples.

In Section \ref{section:J12action}, we show how the Windowing conjecture allows for a complete solution for the multiplication action of $J_{21}$ in the Jack basis. This was determined by Stanley in the special case of $c=1$.

In Section \ref{section:factorization}, we conjecture an extension of the factorization result of King, Tollu, Toumazet {\cite{King:2009}} to the case of Jacks. This allows us to extend our previous solution to cover a large 7-parameter family of $c=2$ cases, which we also conjecture a general solution to.

This paper provides background for the follow up work \cite{Mickler:2025}, in which the specific solution \eqref{eq:generalsolution} given in this paper is explored in greater detail.




\section{Preliminaries}\label{section:prelim}

\setlength{\epigraphwidth}{0.5\textwidth}
\renewcommand{\epigraphsize}{\tiny}


\ytableausetup{boxsize=0.8em}

\subsection{Definitions} 
We use the same notation and conventions as \cite{Alexandersson:2023}, except where stated otherwise. For a partition $\mu$ of size $n$, let $J_\mu$ denote the associated Jack symmetric function with parameter $\alpha \in \setR$, normalized so that the coefficient of the power sum variable $p_1^{n}$ in $J_\mu$ is $1$. These form a basis of the ring of symmetric functions. In the $\alpha \to 1$ limit, the so called \emph{Schur} limit, we have $J_\mu \to h_\mu\, s_\mu$, where $h_\mu := \prod_{b\in \mu} h_\mu(b)$ is the product of all hook lengths of the boxes in $\mu$, and $s_\mu$ is the Schur symmetric function. The Jack Littlewood-Richardson coefficients $g_{\mu\nu}^{\lambda}$ are defined by the expansion of a product of Jack functions
\begin{equation} J_\mu \cdot J_\nu = \sum_{\lambda} g_{\mu\nu}^{\lambda} J_\lambda. \end{equation}
For example, 
\begin{equation}
\jackjLR_{21,21}^{321} = \frac{6\alpha(2+11\alpha+2\alpha^2)}{(2+\alpha)(1+2\alpha)(3+2\alpha)(2+3\alpha)}.
\end{equation}
The Schur Littlewood-Richardson coefficients are recovered in the limit
\begin{equation}\label{eq:schurLRjack}
c_{\mu\nu}^{\lambda} =  \frac{  h_{\lambda}{} }{h_{\mu}{} h_{\nu}{}  }  \left(  \jackjLR_{\mu\nu}^\lambda    \right)_{\alpha \to 1}.
\end{equation}
We use the following standard notation for the $\alpha$-Hall norm of the Jacks,
\begin{equation}\label{eq:jacknorm}
j_\lambda :=  \norm{\jackj_\lambda}^2 = \prod_{b\in \lambda} \lh{\lambda}{b} \, \uh{\lambda}{b}.
\end{equation}
where $\ulh{\lambda}{b}$ are upper/lower hook lengths, and the equality is given by a theorem of Stanley.

\subsection{The Stanley Conjectures}\label{sect:stanley}

In \cite{Stanley:1989}, several remarkable conjectures were made regarding the 
structure of Jack Littlewood--Richardson coefficients. 

\begin{conjecture}[{The Stanley Conjecture, \cite[8.3]{Stanley:1989}}]\label{conj:stanley} The \emph{\bf Stanley coefficients} $g_{\mu\nu\lambda} := \jackjLR_{\mu\nu}^\lambda \cdot j_\lambda$ are non-negative integer polynomials in $\alpha$, i.e. 
\begin{equation*}
g_{\mu\nu\lambda} \in \setZ_{\geq0}[\alpha].
\end{equation*}
\end{conjecture}
\begin{example}
We have $g_{21,21,321} = 6\alpha^4(2+\alpha)(1+2\alpha)(2+11\alpha+ 2\alpha^2)$.
\end{example}

In general, we use the terminology of a \emph{strong} form of the Stanley conjecture to refer to any conjecture that proposes an explicit form for $\jackjLR_{\mu\nu}^\lambda$ which \emph{manifestly} exhibits the non-negativity of Conjecture~\eqref{conj:stanley}. Stanley conjectured such a form in the case $c_{\mu\nu}^\lambda =1$.

\begin{conjecture}[{Strong Stanley Conjecture, see \cite[8.5]{Stanley:1989}}]\label{conj:strongstanley}
If $c_{\mu\nu}^\lambda =1$, then the corresponding Jack LR coefficient has the form
\begin{equation*}\label{formula:strongstanley}
\jackjLR_{\mu \nu}^{\lambda} = \frac{
\prod_{b \in \mu} {h}^{\bullet}_{\mu}(b) \,\,
\prod_{b \in \nu} {h}^{\bullet}_{\nu}(b) }{
\prod_{b \in \lambda} {h}^{\bullet}_{\lambda}(b) },
\end{equation*}
where $\bullet$ indicates either $\sU$ or $\sL$ chosen for each box $b$ in $\mu,\nu$ and $\lambda$.
Moreover, one chooses $\lh{\,}{b}$ (and hence $\uh{\,}{b}$) the same number of times in the numerator and denominator.
\end{conjecture}

Note that the non-negativity follows immediately as each hook factor $\ulh{\,}{b}$ is a non-negative linear polynomial in $\alpha$. Motived by this conjecture, we make the following definition,
\begin{definition}
For any triple of partitions $\mu,\nu, \lambda$, a \emph{\bf{Stanley Diagram}} $\bm{D} \in \StD_{\mu,\nu}^{\lambda}$ is an assignment $b \to \bmD_b \in \{\sU,\sL\}$ for every box $b$ in each of $\mu,\nu$ and $\lambda$.  We draw such assignments as a ratio e.g.

\[ 
\ytableausetup{boxsize=0.8em}
\bm{D} =  \frac{ \begin{ytableau}
*(boxU) \sU \\
*(boxL)  \sL  &  *(boxL) \sL
\end{ytableau}\,\,\begin{ytableau}
*(boxU) \sU  \\
*(boxL)  \sL &  *(boxU) \sU
\end{ytableau} }{ \begin{ytableau}
*(boxU) \sU \\
*(boxL) \sL  & *(boxU) \sU \\
*(boxL) \sL & *(boxU)\sU  & *(boxL) \sL
\end{ytableau} } \in \StD_{21,21}^{321}
\ytableausetup{boxsize=0.8em}
\] 
Note that the hook choices for the boxes in $\lambda$ are in the denominator. 

For a Diagram $\bm{D}$, we define the \emph{\bf{weight}} $|\bm{D}|$ to be the number of of assignments of lower hooks $\sL$ in the numerator minus the number in the denominator, i.e.
\begin{equation} 
|\bm{D}| := \#\{\bm{D}_b = \sL : b \in \mu,\nu \} - \#\{\bm{D}_b = \sL : b \in \lambda \}.
\end{equation}
\end{definition}


These diagrams capture the assignment of an upper or lower hook \emph{symbol} to a particular box in a diagram, which is more general than considering only the hook \emph{lengths} corresponding to such an assignment. We now formalize this distinction.

For a box $b$ in the tableau associated to a partition $\lambda$, we will consider the assignment of an upper or lower hook to be denoted\footnote{In \cite{Alexandersson:2023}, this notation is used instead to denote the hook \emph{vector}, a closely related concept.} by the hook \emph{symbols} of either $\uhc{\lambda}{b}$ or $ \lhc{\lambda}{b}$, and the hook \emph{lengths} (i.e. the $\alpha$-content of that hook) to be denoted $\uh{\lambda}{b}$ or $\lh{\lambda}{b}$. Inside a tableau, we denote upper hook symbol by $\begin{ytableau}*(boxU) \sU \end{ytableau}$ and lower hook symbol by $\begin{ytableau}*(boxL) \sL \end{ytableau}$. For example, for the partition $\lambda = 532$ and the box $b=(1,0)$,
\begin{equation}
\begin{ytableau} 
{}& \\
& & \\ 
& b& & & 
\end{ytableau}
\end{equation}
The upper hook \emph{symbol} is
\begin{equation} \uhc{\lambda}{b} =
\begin{ytableau} 
{}& \\
& & \\ 
& *(boxU) \sU & & &
\end{ytableau}
\end{equation}
whereas the upper hook \emph{length} is $\uh{\lambda}{b}= 2+4\alpha$. The $\alpha$-\emph{\bf{evaluation}} map $[ \cdot  ]$, is defined by sending hook symbols to their corresponding hook lengths\footnote{One could instead consider the Macdonald evaluation, e.g.
\[ [\uhc{\lambda}{b}]_{q,t} \coloneqq 1-q^{\arm(b)+1} t^{\leg(b)}. \]}. That is
\[ [  \uhc{\lambda}{b} ] := \uh{\lambda}{b} = \leg(b) + (\arm(b)+1)\alpha.\]

For a hook assignment $\sA \in \{\sU,\sL\}$, we let $\bar \sA$ denote the opposite assignment.

We extend the evaluation map to act on Stanley diagrams so that the the result is a rational function in $\alpha$, with hooks for $\mu,\nu$ in the numerator, and those for $\lambda$ in the denominator.

\begin{conjecture}[{Strong Stanley Conjecture} - rephrased]\label{conj:strongstanley2}
If $c_{\mu\nu}^\lambda =1$, then there exists a Stanley diagram $\bm{D}\in \StD_{\mu\nu}^{\lambda}$ of weight $|\bm{D}| = 0$, such that the corresponding Jack LR coefficient has the form
\begin{equation*}\label{formula:strongstanley2}
\jackjLR_{\mu \nu}^{\lambda} = [\bm{D}].
\end{equation*}
\end{conjecture}


\subsection{Structural Conjecture}

In the closing section of \cite{Stanley:1989}, the author contemplates whether an expression of the type of Conjecture \eqref{conj:strongstanley} may also hold for cases where $c_{\mu\nu}^\lambda \geq 1$, where one might instead sum over $c_{\mu\nu}^\lambda \in \BZ_{> 0}$ Stanley diagrams:

\begin{quotation} 
\emph{One possibility is the following: $\jackjLR_{\mu\nu\lambda}$ can be written as a sum of $c_{\mu\nu}^\lambda$ expressions of the form [of the c=1 case], each possibly multiplied by a power of $\alpha$. \,\, {\cite[p114]{Stanley:1989}}}
\end{quotation}

 However, a counter example was presented that halted this idea (see example \ref{example:hanlonpositivity}). In this paper we pursue this train of thought, and we will build evidence for the following straightforward generalization to $c_{\mu\nu}^\lambda \geq 1$.

\begin{definition}
A \emph{{\bf{Stanley Sum}}} $\bmss  = \sum_{\bm{D}\, \in \StD_{\mu\nu}^{\lambda}} c_\bmD\,  \bm{D} \in \StS_{\mu\nu}^{\lambda}$, with $c_\bmD \in \BZ$, is an element in the $\setZ$-span of Stanley Diagrams. The evaluation map $[\cdot]$ extends linearly to such sums. Similarly, we define the \emph{\bf{weight}} of such a sum $\bmss$ to be 
\begin{equation}
|\bmss| := \sum_{\bm{D} \in \StD_{\mu\nu}^{\lambda}} c_\bmD \, |\bm{D}| \,\,\in  \BZ.
\end{equation}
\end{definition}

\begin{conjecture}[General Structure]\label{conj:generalstructure}
For each triple of partitions $\mu,\nu,\lambda$, there exists a Stanley sum, 
\begin{equation}
\JackjLR_{\mu\nu}^{\lambda} = \sum_{\bm{D} \in \StD_{\mu\nu}^{\lambda}} c_\bmD \, \bm{D}, \qquad c_D \in \BZ,
\end{equation}
of weight $|\JackjLR_{\mu\nu}^{\lambda}| = 0$, such that the corresponding Jack LR coefficient is recovered by evaluation.
\begin{equation}
\jackjLR_{\mu\nu}^{\lambda} = [\JackjLR_{\mu\nu}^{\lambda}].
\end{equation}

We call such a Stanley sum $\JackjLR_{\mu\nu}^{\lambda}$ a \emph{\bf{Rule}} for $\jackjLR_{\mu\nu}^{\lambda}$.
\end{conjecture}

Some comments regarding this conjecture.
First, in the Schur limit $\alpha\to1$, we find that $[\bm{D}] \to h_\mu h_\nu / h_\lambda $ for all $\bm{D}$, and $\jackjLR_{\mu\nu}^{\lambda} \to c_{\mu\nu}^{\lambda} \cdot h_\mu h_\nu / h_\lambda $, and so for a Rule we have
\begin{equation}\label{eq:schurreduction}
 \sum_{\bm{D} \in \StD_{\mu\nu}^{\lambda}} c_\bmD = c_{\mu\nu}^{\lambda}.
\end{equation}
Thus, if we expect to see more than the number $c_{\mu\nu}^{\lambda}$ of diagrams contribute to such a Rule, then we must allow for the possibility of some negative coefficients $c_\bmD$.

Secondly, we observe that our Conjecture \ref{conj:generalstructure} is compatible with the Stanley Conjecture \ref{conj:stanley}, in that the poles of $[\JackjLR_{\mu\nu}^{\lambda}]$ as a meromorphic function of $\alpha$ can only be of the form $1/\ulh{\lambda}{b}$, that is, 
\begin{equation}
j_\lambda \cdot [\JackjLR_{\mu\nu}^{\lambda}]  \in \BZ[\alpha].
\end{equation}
However, Conjecture \ref{conj:generalstructure} is much weaker in the sense that it does not suggest any positivity property of the coefficients $c_\bmD$ (indeed, we expect them to \emph{not} be non-negative). We do however make a modest refinement to this positivity aspect later in Conjecture \ref{conj:windowpositivity}.

Lastly, we observe that Conjecture \ref{conj:generalstructure} is not very constraining, as there are such a large number of terms possible in the right hand sum, and we are only asking for an equality of a single rational function, its not surprising if at least one solution exists. However, we will soon introduce two types of compatibility relations between these Rules for different $\lambda,\mu$ and $\nu$ that greatly restrict this abundance of degrees of freedom. But first, we review some known cases of Jack LR coefficients.




\subsection{Known Cases}

Here we collect some known cases of the Strong Stanley conjecture \eqref{conj:strongstanley}, that is, in all the cases below the corresponding Schur LR coefficient satisfies $c=1$.

\begin{theorem}[{Pieri rule, \cite[6.1]{Stanley:1989}}]\label{thm:pieri}
Let $\lambda/\mu$ be a horizontal $r$-strip, then the Jack LR coefficient is given by the following expression
\begin{equation*}
\jackjLR_{\mu,(r)}^\lambda = \frac{
 \prod_{b \in \mu} h^{A(\lambda/\mu,b)}_\mu(b) \cdot
 \prod_{b \in (r)}  \uh{(r)}{b}  }{
 \prod_{b \in \lambda} h^{A(\lambda/\mu,b)}_\lambda(b)  },
\end{equation*}
where
\begin{equation*}
A(\lambda/\mu,b) \coloneqq \begin{cases}
\sU & \text{ if $\lambda/\mu$ contains a box in the same column as $b$} \\
\sL  &  \text{ otherwise.}
\end{cases}
\end{equation*}

\end{theorem}

\begin{example}\label{ex:pieri} 
\begin{equation*}
\JackjLR_{75411,3}^{76431} =
\ytableausetup{boxsize=0.8em}
\frac{
\begin{ytableau}
*(boxL) \sL \\
*(boxL) \sL \\
*(boxL) \sL & *(boxU) \sU  & *(boxU) \sU  & *(boxL) \sL \\
*(boxL) \sL & *(boxU) \sU  & *(boxU) \sU  & *(boxL) \sL & *(boxL) \sL \\
*(boxL) \sL & *(boxU) \sU  & *(boxU) \sU  & *(boxL) \sL & *(boxL) \sL  & *(boxU) \sU & *(boxL) \sL
\end{ytableau}
\quad
\begin{ytableau}
*(boxU) \sU & *(boxU) \sU & *(boxU) \sU
\end{ytableau}}{
\begin{ytableau}
*(boxL) \sL \\
*(boxL) \sL & *(boxU) \sU  & *(boxU) \sU \\
*(boxL) \sL & *(boxU) \sU  & *(boxU) \sU  & *(boxL) \sL \\
*(boxL) \sL & *(boxU) \sU  & *(boxU) \sU  & *(boxL) \sL &*(boxL) \sL  & *(boxU) \sU \\
*(boxL) \sL & *(boxU) \sU  & *(boxU) \sU  & *(boxL) \sL &*(boxL) \sL  & *(boxU) \sU & *(boxL) \sL
\end{ytableau}} \\
\end{equation*}
\end{example}


Note however, there is an ambiguity in the Pieri rule.
\begin{example}We consider two distinct Stanley diagrams that evaluate to $\jackjLR_{1,2}^{21}$, both examples of the Pieri rule
\begin{equation} \JackjLR_{2,1}^{21} = 
 \frac{ \begin{ytableau}
*(boxU) &  *(boxL) \\
\end{ytableau} \,\, \begin{ytableau}
 *(boxU)
\end{ytableau} } { \begin{ytableau}
 *(boxU) \\
*(boxU)  &  *(boxL)
\end{ytableau} },\quad
\JackjLR_{1,2}^{21} = 
 \frac{  \begin{ytableau}
  *(boxU)
\end{ytableau}\,\,\begin{ytableau}
*(boxU) &  *(boxU) \\
\end{ytableau}  } { \begin{ytableau}
 *(boxU) \\
 *(boxU)  &  *(boxU)
\end{ytableau} }.
\end{equation}

\end{example}

Stanley also generalized the above Pieri rule to the case of Maximal Filling tableaus.

\begin{theorem}[Maximal Filling Case {\cite[8.6]{Stanley:1989}}]
Suppose that $c_{\mu\nu}^{\lambda} =1$, and the unique $LR$-tableau $T$ of shape $\lambda/\mu$ with weight $\nu$ has the \emph{\bf{maximal filling}} property, that is, that every column $C$ of $T$ consists of the sequential integers $1,2, \ldots, n_C$.
Let $r_i$ be defined as the largest entry of $T$ in row $i$ of $\lambda/\mu$, and let $c_j$ be the largest entry of $T$ in column $j$ of $\lambda/\mu$ (or $0$). For $b=(j,i)$ define
\begin{equation}\label{eq:maxfillingrule}
A(T,b)  \coloneqq
\begin{cases}
\sU & \text{ if $r_i\leq c_j$ and $c_j>0$} \\
\sL &  \text{ otherwise},
\end{cases}
\end{equation}
For $b=(j,i) \in \mu$, let $b_* \coloneqq (j,i+c_j) \in \lambda$. The Jack LR coefficient is given by

\begin{equation}\label{eq:maxfilling}
g_{\mu\nu}^{\lambda} = \frac{ \prod_{b \in \mu} h^{A(T,b_*)}_{\mu}(b) \cdot \prod_{b\in\nu} \uh{\nu}{b} }{\prod_{b \in \lambda} h^{A(T,b)}_{\lambda}(b) }.
\end{equation}
\end{theorem}

I.e. the diagram for which $\bmD= \JackjLR_{\mu\nu}^{\lambda}$ has the property that for each $b \in \mu$, corresponding to $b_* \in \lambda$, we have $\bmD_{b} = \bmD_{b_*}$.

\begin{example}[{\cite[p113-114]{Stanley:1989}}]\label{ex:stanleytableau}
Take $\lambda = 766654211$, $\mu = 7553322$, $\nu=542$. The unique $LR$ tableau $T$ of shape $\lambda/\mu$ and weight $\nu$ is given by 
\[ \ytableausetup{boxsize=0.8em}
T = \begin{ytableau}
2  \\
1 \\
 &   \\
 &  & 1 & 3 \\
 &  &   & 2 & 2 \\
 &   &  & 1&1 & 3 \\
 &   &   & &&2\\
 &    &    &   &   & 1 \\
  &  &  &  &   &  &  \\
\end{ytableau}
\]
We see that $T$ has the maximal filling property, with vectors $c = \{2,0,1,3,2,3,0\}$, and $r = \{0,1,2,3,2,3,0,1,2\}$. Using the formula above, we find the following expression for the rule

\begin{equation*}
\JackjLR_{\mu,\nu}^{\lambda} =
\ytableausetup{boxsize=0.8em}
\frac{
\begin{ytableau}
*(boxU) & *(boxL)   \\
*(boxU) & *(boxL)   \\
*(boxU) & *(boxL)  & *(boxL)  \\
*(boxL) & *(boxL)  & *(boxL)   \\
*(boxU) & *(boxL)  & *(boxL)  & *(boxU) & *(boxU)  \\
*(boxL) & *(boxL)  & *(boxL)  & *(boxU) & *(boxL)  \\
*(boxU) & *(boxL)  & *(boxU)  & *(boxU) & *(boxU) & *(boxU) & *(boxL) \\
\end{ytableau} \quad
\begin{ytableau}
*(boxU)  &*(boxU)   \\
*(boxU)  &*(boxU)  &*(boxU)  &*(boxU)     \\
*(boxU)  &*(boxU)  &*(boxU)  &*(boxU)  &*(boxU)  
\end{ytableau}}{
\begin{ytableau}
*(boxU)  \\
*(boxU)  \\
*(boxU) & *(boxL)   \\
*(boxL) & *(boxL)  & *(boxL)  & *(boxU) \\
*(boxU) & *(boxL)  & *(boxL)  & *(boxU) & *(boxU)\\
*(boxL) & *(boxL)  & *(boxL)  & *(boxU) & *(boxL) & *(boxU) \\
*(boxU) & *(boxL)  & *(boxL)  & *(boxU) & *(boxU) & *(boxU) \\
*(boxU) & *(boxL)  & *(boxU)  & *(boxU) & *(boxU) & *(boxU) \\
*(boxU) & *(boxL)  & *(boxU)  & *(boxU) & *(boxU) & *(boxU) & *(boxL)\\
\end{ytableau}}  \\
\end{equation*}
\end{example}

\begin{theorem}[Rectangular case, see {\cite[4.7]{CaiJing:2013}}]\label{thm:stanleyrectangular}
Let $\mu \subseteq m^n$, and $\bar\mu$ be the partition that is the reverse\footnote{I.e. rotate the shape 180 degrees.} of the shape $m^n/\mu$. Then the strong Stanley conjecture \eqref{conj:strongstanley} holds for $\jackjLR_{\mu,\bmu}^{m^n}$, in particular
\begin{equation}\label{eq:stanleyrectangular}
\jackjLR_{\mu,\bmu}^{m^n}  = \frac{
\prod_{b \in \mu}  \lh{\mu}{b}  \,\,
 \prod_{b \in \bmu}  \uh{\bmu}{b} }{
 \prod_{b \in m^n}  h^{D(\mu,m^n,b)}_{m^n}(b) },
\end{equation}
where, for $b=(b_1,b_2)\in m^n$, we define
\begin{equation}\label{defn:rectD}
D(\mu,m^n,b) \coloneqq
\begin{cases}
\sL & \text{ if $(b_1,n-1-b_2) \in \mu$} \\
\sU &  \text{ otherwise}.
\end{cases}
\end{equation}
\end{theorem}
\begin{example}

\begin{equation*}
\ytableausetup{boxsize=0.8em}
\JackjLR_{211,221}^{333} = \frac{
\begin{ytableau}
*(boxL) \sL   \\
*(boxL) \sL   \\
*(boxL) \sL & *(boxL) \sL
\end{ytableau}\quad
\begin{ytableau}
*(boxU) \sU  \\
*(boxU) \sU & *(boxU) \sU  \\
*(boxU) \sU & *(boxU) \sU
\end{ytableau}}{\begin{ytableau}
*(boxL) \sL & *(boxL) \sL & *(boxU) \sU \\
*(boxL) \sL & *(boxU) \sU & *(boxU) \sU  \\
*(boxL) \sL & *(boxU) \sU & *(boxU) \sU
\end{ytableau}}, \qquad 
\JackjLR_{221,211}^{333} = \frac{
\begin{ytableau}
*(boxL) \sL  \\
*(boxL) \sL & *(boxL) \sL  \\
*(boxL) \sL & *(boxL) \sL
\end{ytableau}\quad\begin{ytableau}
*(boxU) \sU   \\
*(boxU) \sU   \\
*(boxU) \sU & *(boxU) \sU
\end{ytableau}}{\begin{ytableau}
*(boxL) \sL & *(boxL) \sL & *(boxU) \sU \\
*(boxL) \sL & *(boxL) \sL & *(boxU) \sU  \\
*(boxL) \sL & *(boxU) \sU & *(boxU) \sU
\end{ytableau}}.
\end{equation*}

\end{example}

Note: Corresponding to the above, the Jack LR cofficients $\jackjLR_{\mu,\bmu}^{m^n}=\jackjLR_{\bmu,\mu}^{m^n}$ are symmetric in the lower arguments (although this is not manifest\footnote{For a demonstration of this property in this case see \cite{Alexandersson:2023}.}), however the corresponding \emph{rules} $\JackjLR_{\mu,\bmu}^{m^n} \neq \JackjLR_{\bmu,\mu}^{m^n}$ are not (i.e. after exchanging the top two factors).


The next example was the motivation for the present work. Wherein the surprising simplicity of the answer, and the compatibility with the Rectangular case \eqref{thm:stanleyrectangular}, merited further investigation.

\begin{theorem}[Rectangular union case, \cite{Alexandersson:2023} Corollary 3.6 ]\label{thm:rectunion}
Let $m^n$ be a rectangular partition, and $\mu$ be any partition (i.e. not necessarily $m^n \subset \mu$). Let $\sigma \coloneqq m^n \cap \mu$, and $\bsigma$ be the partition given by the reverse of $m^n / \sigma$. Then, the Strong Stanley Conjecture~\eqref{conj:strongstanley} holds for $\jackjLR_{\mu,\bsigma}^{\mu\cup m^n}$.  

In particular, let $R \subset m^n$ be the smallest rectangular shape such that $\mu$ and $\mu\cup m^n$ agree outside $R$. Say $R$ is of shape $k^\ell$. Let $\mu_R$ be the partition giving by shifting the shape $\mu \cap R$ back to the origin, and $\bar \mu_R$ be the reverse of $k^\ell / \mu_R$. We then have the following decomposition for the LR coefficient
\begin{equation}\label{eq:rectunion}
\jackjLR_{\mu,\bsigma}^{\mu\cup m^n} = \raisebox{-0.5\height}{\includegraphics[scale=0.5,page=1]{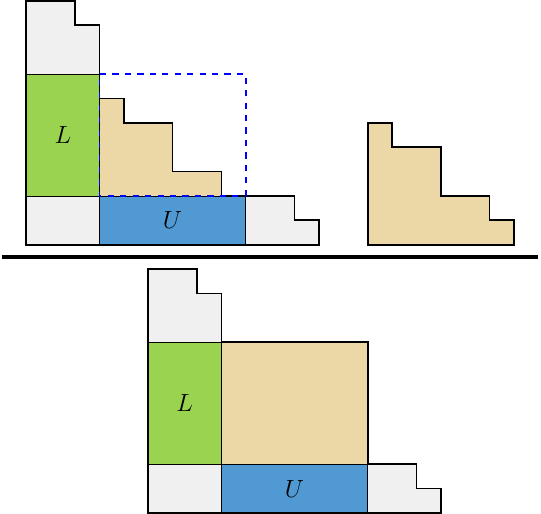}} \times  \jackjLR_{\mu_R, \bar \mu_R}^{k^\ell}
\end{equation}
where the left factor (henceforth denoted $F_R$)  is a products of $L$-hooks for every box to the left of $R$ (denoted by dashed line), 
and $U$-hooks for every box below it, for both $\mu$ in the numerator 
and $\mu \cup m^n$ in the denominator. For boxes sharing neither a row or column with a box in $R$, we can assign \emph{either} upper or lower hooks as these cancel between the numerator and denominator, and so we indicate this ambiguity with light grey \begin{ytableau}*(boxG)\end{ytableau}. The second factor $\jackjLR_{\mu_R, \bar \mu_R}^{k^\ell}$ is covered by the rectangular case above.

\end{theorem}
Thus we see that the rule for the Rectangular union case, $\JackjLR_{\mu,\bsigma}^{\mu\cup m^n}$, is equivalent to starting with the rule for the Rectangular case \eqref{thm:stanleyrectangular} inside the `window' $R$, given by $\JackjLR_{\mu_R, \bar \mu_R}^{k^\ell}$, and extending outside of this window by $F_R$.
\begin{example}\label{ex:rectunion}
\begin{equation*}
\JackjLR_{42211,211}^{43331} =
\ytableausetup{boxsize=0.8em}
\frac{
\begin{ytableau}
  *(boxG)  \\
*(boxL) \sL   \\
*(boxL) \sL & *(boxS)  \\
*(boxL) \sL & *(boxS)  \\
  *(boxG)& *(boxU) \sU & *(boxU) \sU & *(boxG)  \\
\end{ytableau}\quad
\begin{ytableau}
*(boxS)   \\
*(boxS)   \\
*(boxS)  & *(boxS) 
\end{ytableau}}{
\begin{ytableau}
  *(boxG)\\
*(boxL) \sL & *(boxS)  & *(boxS)  \\
*(boxL) \sL & *(boxS)  & *(boxS)  \\
*(boxL) \sL & *(boxS)  & *(boxS)  \\
 *(boxG) & *(boxU) \sU & *(boxU) \sU & *(boxG)  \\
\end{ytableau}} \circ (\JackjLR_{11,211}^{222})_R \\
\end{equation*}
The notation $(\JackjLR)_R$ denotes taking hook choices for the boxes inside $R \cap\mu$, $R \cap\lambda$ and $\nu$ to agree with those choices of the rule $\JackjLR$ inside the brackets.
\end{example}




\section{Windowing Conjecture}\label{section:windowing}

\setlength{\epigraphwidth}{0.6\textwidth}
\epigraph{You rows of houses! you window-pierc'd fa\c{c}ades! you roofs!\\
You porches and entrances! you copings and iron guards!\\
You windows whose transparent shells might expose so much!\\
You doors and ascending steps! you arches!\\
You gray stones of interminable pavements! you trodden crossings!\\
From all that has touch'd you I believe you have imparted to yourselves, and now would impart the same secretly to me...}{\textit{Walt Whitman \\ Song of the Open Road}}

Motivated by the Rectangular Union case~\eqref{thm:rectunion}, we introduce the following constructions.

\begin{definition}
A \emph{\bf{window}} is a shape $R$ that is closed under meets and joins of its boxes (hence, necessarily a union of rectangles). 

For a partition $\mu$, let $\mu_R$ be the \emph{\bf{windowed partition}} obtained by `viewing' $\mu$ through the window $R$. That is, $\mu_R$ is given by the shape $\mu \cap R$ joined along edges of $R$ and re-centered back to the origin as a partition. We color the boxes of windows with beige \begin{ytableau}
 *(boxS)  
\end{ytableau}.

For any shape $\sigma$, define the \emph{\bf{meet-join}} of $\sigma$, denoted $\meetjoin(\sigma)$, to be the window that is given by the closure of $\sigma$ under meets and joins of its boxes. In particular, for $\mu\subset \lambda$ we often consider $\meetjoin(\lambda/\mu)$, which is the smallest window such that $\mu$ and $\lambda$ agree outside of it. 
\end{definition}

\begin{example}\label{ex:windowing}
For $\lambda= 44322$, and the window $R$ given below, we have $\lambda_R = 3211$.
\begin{equation*}
\ytableausetup{boxsize=0.7em}
\lambda \coloneqq
\begin{ytableau}
{}&   \\
&    \\
 &   &  \\
  &  &  &   \\
&   &   &\\
\end{ytableau}  \\
\qquad
R \coloneqq
\begin{ytableau}
{}& *(boxS)   & *(boxS)  & *(boxS) \\
& *(boxS)   & *(boxS)  & *(boxS) \\
 & *(boxS)  & *(boxS) & *(boxS)\\
  &  &  &   \\
& *(boxS)   &*(boxS)   & *(boxS) \\
\end{ytableau}  \\
\qquad
\lambda_R = \begin{ytableau}
{}   \\
   \\
&  \\
& & \\
\end{ytableau} 
\end{equation*}
\end{example}

For further demonstration, let $\mu \subset \lambda$ be such that $\mu$ and $\lambda$ agree (as shapes) outside of two disjoint\footnote{And which can be totally separated horizontally and vertically.} rectangles $R_1, R_2$. Let $R_3 = \bmin{R_1}  {R_2}$ and $R_4 = \bmax{R_1}  {R_2}$. 
\begin{equation}\label{fig:twowindowed}
 {\includegraphics[scale=0.8,page=3]{stanleyConjecturePics}} 
 \end{equation}

Viewing through the window $R\coloneqq R_1 + R_2 + R_3 + R_4$, the shapes are joined on their edges, defining the two windowed partitions $\lambda_R$ and $\mu_R$.
\[  {\includegraphics[scale=0.8,page=4]{stanleyConjecturePics}}  \]

We proceed by noting that there is a compatibility between windowing and Schur LR coefficients.

\begin{lemma}[Schur Windowing]\label{lemma:schurlrwindow}
For any $\mu \subset \lambda$, for all $\nu$, and $R$ a window with $\lambda/\mu \subseteq R$, we have equality of the Schur Littlewood-Richardson coefficients
\begin{equation}
c_{\mu,\nu}^{\lambda} = c_{\mu_R,\nu}^{\lambda_R}.
\end{equation}
\end{lemma}
\begin{proof}
Follows from the obvious bijection of LR tableau of skew-shape $\lambda/\mu$ and $\lambda_R/\mu_R$, or equivalently, from the equality of skew Schur functions $s_{\lambda/\mu} = s_{\lambda_R/\mu_R}$.
\end{proof}

In effect, this property says that we can translate the connected components of $\lambda/\mu$ around and the Schur LR coefficient will be invariant. Our goal is to generalise this `translation invariance' to the case of Jack LR coefficients. Following from Formula \eqref{eq:rectunion}, we make the following key definitions.

\begin{definition}
Let $\mu \subset \lambda$ be such that $\mu$ and $\lambda$ agree (as shapes) outside of a window $R$, i.e. $\meetjoin({\lambda/\mu}) \subset R$.

Define the \emph{\bf{window factor}} for $R$, denoted $\bm{F}_{\mu,R}^\lambda$, as the ratio of hook symbols given by the following product over all boxes in $\mu$ and $\lambda$ which are outside $R$, 
\begin{equation}
\bm{F}_{\mu,R}^\lambda \coloneqq \prod_{b \in \mu/R} \frac{ \ula{\mu}{b} }{ \ula{\lambda}{b}  },
\end{equation}
where for those boxes that:
\begin{itemize}
\item share a row with a box in $R$, we assign a lower hook symbol \begin{ytableau}*(boxL)\end{ytableau}.
\item share a column with a box in $R$, we assign a upper hook symbol \begin{ytableau}*(boxU)\end{ytableau}.
\item share neither a row or column with a box in $R$, then in this case the hook lengths for this box in $\mu$ and $\lambda$ agree, and thus they will cancel out when the hook symbols are evaluated. Thus, we can assign either an upper or lower hook symbol (or neither), an ambiguity indicated by the neutral symbol \begin{ytableau}*(boxG)\end{ytableau}. 
\end{itemize}

For example, if we consider the window given by Figure \eqref{fig:twowindowed}, we have
\begin{equation}
 \bm{F}_{\mu,R}^\lambda :=   \raisebox{-0.5\height}{\includegraphics[scale=0.8,page=5]{stanleyConjecturePics}}  
 \end{equation}

Correspondingly, let $F_{\mu,R}^\lambda \coloneqq [\bm{F}_{\mu,R}^\lambda]$ be the evaluation that sends hook symbols to hook lengths. Often we just write $\bm{F}_{R} \equiv \bm{F}_{\mu,R}^\lambda$ if the context of the partitions is clear. Note that the window factor always has weight $|\bm{F}| = 0$.
\end{definition}

\begin{example}
In the earlier example \eqref{ex:rectunion}, with $\lambda=43331$ and $\mu=42211$, we have
\begin{equation*}
\ytableausetup{boxsize=0.8em}
R := \begin{ytableau}
{} &   & &  \\
 & *(boxS)  & *(boxS) &   \\
 & *(boxS)  & *(boxS) &  \\
 & *(boxS)  & *(boxS) &\\
 &  & &   \\
\end{ytableau}
\quad\quad\quad
\bm{F}_{\mu,R}^\lambda =
\frac{
\begin{ytableau}
  *(boxG)  \\
*(boxL)    \\
*(boxL)  & *(boxS)  \\
*(boxL)  & *(boxS)  \\
  *(boxG)& *(boxU)  & *(boxU)  & *(boxG)  \\
\end{ytableau}}{
\begin{ytableau}
  *(boxG)\\
*(boxL) & *(boxS)  & *(boxS)  \\
*(boxL) & *(boxS)  & *(boxS)  \\
*(boxL) & *(boxS)  & *(boxS)  \\
 *(boxG) & *(boxU) & *(boxU) & *(boxG)  \\
\end{ytableau}} \\
\end{equation*}
\end{example}

\begin{lemma}
If $\lambda/\mu \subseteq R' \subset R$, then $\bm{F}_{\mu,R'}^\lambda$ agrees with $\bm{F}_{\mu,R}^\lambda$ outside of $R$. Thus any window containing $\lambda/\mu$ is compatible with the smallest such window $\meetjoin(\lambda/\mu)$.
\end{lemma}
\begin{proof}
Any box $b$ outside of $R$ that lies in a row (or column) that intersects $R'$ but not $R$ has the property that $\ulh{\mu}{b}= \ulh{\lambda}{b}$, and hence there is a cancellation between this hook length in the numerator and denominator, and thus we assign this box the neutral symbol.
\end{proof}

\begin{definition}
For a window $R$, define the \emph{\bf{de-windowing}} operator $\left( \cdot \right)_R$ by its action on hook symbols as
\begin{equation}
\left( \bm{h}^\bullet_{\lambda_R}(b)\right)_R \coloneqq \bm{h}^\bullet_{\lambda}(b^R),
\end{equation}
where the box $b\in \lambda_R $ corresponds to $b^R \in \lambda \cap R$. This operator is extended by linearity to act on Stanley sums.

For ease of notation, we write de-windowing followed by evaluation as
\[ [ \bmD ]_R \coloneqq [\,( \bmD )_R \,]. \]
\end{definition}

\begin{example}In the case of Example \eqref{ex:windowing}, we have
\begin{equation*}
\ytableausetup{boxsize=0.8em}
\left( \begin{ytableau}
*(boxU) \sU  \\
*(boxL) \sL    \\
*(boxL) \sL  &  *(boxU) \sU \\
 *(boxU) \sU &  *(boxL) \sL &  *(boxL) \sL \\
\end{ytableau} \right)_R = 
\begin{ytableau}
{}& *(boxU) \sU   \\
&  *(boxL) \sL    \\
 & *(boxL) \sL  &  *(boxU) \sU  \\
  &  &  &   \\
& *(boxU) \sU  &  *(boxL) \sL  &  *(boxL) \sL \\
\end{ytableau} 
\end{equation*}
\end{example}

Motivated by the Rectangular Union case \eqref{thm:rectunion}, we propose that a generalisation of the Schur Windowing property \eqref{lemma:schurlrwindow} holds in general for all Jack LR coefficients. 

\begin{conjecture}[Jack Windowing]\label{conj:jackwindowing}
Assume the General Structure Conjecture \eqref{conj:generalstructure} holds. The Rules $\JackjLR_{\mu\nu}^{\lambda}$ satisfy the \emph{\bf{windowing}} property: For a window $R$ with $\lambda/\mu \subset R$, there exists a Rule $\JackjLR_{\mu_R,\nu}^{\lambda_R}$ (i.e. satisfying $[\JackjLR_{\mu_R,\nu}^{\lambda_R}] = \jackjLR_{\mu_R,\nu}^{\lambda_R}$) such that
\begin{equation}
\JackjLR_{\mu,\nu}^{\lambda} :={\bm{F}_{\mu,R}^\lambda} \circ \left(\JackjLR_{\mu_R,\nu}^{\lambda_R}\right)_{R}.
\end{equation}
is a Rule for $\jackjLR_{\mu,\nu}^{\lambda}$.
\end{conjecture}

According to this conjecture, to prescribe a Stanley sum for a given Jack LR coefficient which is constrained by a window $R$, we fill in the boxes outside of $R$ in the diagrams with the window factor $\bm{F}_R$, and inside $R$ there exists a rule $\JackjLR_{\mu_R,\nu}^{\lambda_R}$ for which we can fill in with the de-windowed Stanley sum (extended by linearity over the sum of terms in the rule). 

We can window over the pair $\mu,\lambda$, or $\nu,\lambda$, and such a formula should hold for each pair. As before, we don't assume the sum $\JackjLR_{\mu,\nu}^{\lambda}$ is symmetric in $\mu$ and $\nu$. However, we are cautious that this non-symmetry might imply a non-compatibility between these two orders of windowing.



\subsection{Special Cases}

In this section, we provide evidence for the Jack Windowing conjecture by considering special cases of it.

Firstly, we consider the case where $\mu \subset \lambda$ are such that $\mu$ and $\lambda$ agree outside of a window $R$ that is a single rectangle. In this case, we have $ [\JackjLR_{\mu_R,\nu}^{\lambda_R}]_R  = \jackjLR_{\mu_R,\nu}^{\lambda_R}$, and the Jack Windowing conjecture reduces to the following result, which we are able to prove follows from the work of Bravi-Gandini \cite{Bravi:2021}.
\begin{theorem}[Locality]\label{thm:locality}
Let $\mu \subset \lambda$ be such that $\lambda/\mu$ is contained inside a rectangular window $R$. We then have
\begin{equation}\label{eq:locality}
 \jackjLR_{\mu,\nu}^{\lambda} = F_{\mu,R}^{\lambda} \times \jackjLR_{\mu_R,\nu}^{\lambda_R}, 
 \end{equation}
where, as before, 
\[ \windfact_{\mu,R}^{\lambda} \coloneqq \raisebox{-0.5\height}{\includegraphics[scale=0.5,page=2]{stanleyConjecturePics}} \]
\end{theorem}
\begin{proof}
We show this follows directly from a result of Bravi-Gandini \cite[Thm 10]{Bravi:2021}. For $\lambda \subset \mu$, let $J_{\lambda/\mu}$ denote the skew Jack function. For $b\in \mu$, define\footnote{In \cite{Bravi:2021}, the notation $c_{\mu,\lambda,s} \equiv h^{\sL}_{\lambda,\mu}(s)$, and $c'_{\lambda,\mu,s} \equiv h^{\sU}_{\lambda,\mu}(s) $ is used.}
\begin{equation}
h^{\sL}_{\mu,\lambda}(b) = \arm_{\mu}(b) \alpha + \leg_{\lambda}(b) + 1, \,\,\, h^{\sU}_{\lambda,\mu}(b) \coloneqq (\arm_{\lambda}(b)+1) \alpha + \leg_{\mu}(b).
\end{equation}
For the window $R$ consisting of a single rectangle, let $\lambda_R$ and $\mu_R$ be the corresponding windowed partitions.
In the language of Bravi-Gandini, the skew shapes $\lambda/\mu$ and $\lambda_R/\mu_R$ \emph{coincide up to translation}, in which case they show the following relation of skew Jack functions
\begin{equation}\label{result:bg}
h^{\sL}_{\mu,\lambda} h^{\sU}_{\lambda,\mu} \cdot J_{\lambda_R/\mu_R} = h^{\sL}_{\mu_R,\lambda_R} h^{\sU}_{\lambda_R,\mu_R} \cdot J_{\lambda/\mu},
\end{equation}
where $h^{\sL}_{\mu,\lambda} \coloneqq \prod_{b \in \mu} h^{\sL}_{\mu,\lambda}(b)$, and $h^{\sU}_{\lambda,\mu} \coloneqq \prod_{b \in \mu} h^{\sU}_{\lambda,\mu}(b)$.
We notice that for a box $b \in \mu \cap R \subseteq \mu$, we have $h^{\sL}_{\lambda,\mu}(b) = h^{\sL}_{ \lambda_R,\mu_R}(b_R)$. Thus, we can cancel factors from all boxes inside $R$, and re-write equation \eqref{result:bg} as
\begin{equation}
J_{\lambda/\mu} = \prod_{b \in \mu/ R} h^{\sL}_{\lambda,\mu}(b) h^{\sU}_{\mu,\lambda}(b) \cdot J_{\lambda_R/\mu_R}.
\end{equation}
As we are interested in LR coefficients, which are given by $g_{\mu\nu}^{\lambda} = \langle J_{\lambda/\mu}/j_\lambda ,J_\nu \rangle$, we write the above in terms of the quantities $J_{\lambda/\mu}/j_\lambda$, 
\begin{equation}\label{eq:bgbigprod}
\frac{J_{\lambda/\mu}}{j_\lambda} = \prod_{b \in \mu/ R} \frac{h^{\sL}_{\lambda,\mu}(b) h^{\sU}_{\mu,\lambda}(b)}{h^{\sL}_{\lambda}(b) h^{\sU}_{\lambda}(b)} \cdot \frac{J_{\lambda_R/\mu_R}}{j_{\lambda_R}}.
\end{equation}
Now, if $b \in \mu/R$ does not share a row or column with $R$, then $\arm_\mu(b) = \arm_\lambda(b)$ and $\leg_\mu(b) = \leg_\lambda(b)$, and hence $h^{\sL}_{\lambda,\mu}(b) = h^{\sL}_{\lambda}(b)= h^{\sL}_{\mu}(b)$, and similarly $h^{\sU}_{\mu,\lambda}(b) = h^{\sU}_{\lambda}(b)= h^{\sU}_{\mu}(b)$. We see that the contribution from those such boxes $b$ to the product in \eqref{eq:bgbigprod} vanishes. Thus, that same product reduces to one over the two regions $R_{below}$ and $R_{left}$, of those boxes $b \in \mu/R$ either below or left of $R$. For a box $b \in R_{left}$, we have $\leg_\mu(b) = \leg_\lambda(b)$, and thus $h^{\sL}_{\mu,\lambda}(b) = h^{\sL}_{\mu}(b)$, and $h^{\sU}_{\lambda,\mu}(b) = h^{\sU}_{\lambda}(b)$. For a box $b \in R_{below}$, we have $\arm_\mu(b)=\arm_\lambda(b)$, and thus $h^{\sL}_{\mu,\lambda}(b) = h^{\sL}_{\lambda}(b)$ and $h^{U}_{\lambda/\mu}(b) = h^{\sU}_\mu(b)$. Thus we have shown that equation \eqref{eq:bgbigprod} reduces to
\begin{equation}
\frac{J_{\lambda/\mu}}{j_\lambda} = \prod_{b \in R_{left}} \frac{h^{\sL}_{\mu}(b) }{h^{\sL}_{\lambda}(b) }\prod_{b \in R_{below}} \frac{ h^{\sU}_{\mu}(b)}{ h^{\sU}_{\lambda}(b)} \cdot \frac{J_{\lambda_R/\mu_R}}{j_{\lambda_R}}.
\end{equation}
We see this factor is precisely $F^\lambda_{\mu,R}$, that is, 
\begin{equation}\label{eq:skeqjackwindow}
\frac{J_{\lambda/\mu}}{j_\lambda} = F^\lambda_{\mu,R} \cdot \frac{J_{\lambda_R/\mu_R}}{j_{\lambda_R}}.
\end{equation}
Taking the inner product with $J_\nu$, we recover the result.
\end{proof}

\begin{corollary}
The Rectangular union case \eqref{thm:rectunion} follows immediately from the Locality property \eqref{thm:locality}. 
\end{corollary}

In \cite{Alexandersson:2023}, the Rectangular Union LR coefficient \eqref{thm:rectunion} was calculated directly using a very different approach that involved an intricate calculation. The corollary above demonstrates a vast simplification over that technique. 

\begin{example}\label{ex:locality} By computation, we can verify the following $c=2$ example
\begin{equation*}
\jackjLR_{421,21}^{4321} =
\ytableausetup{boxsize=0.8em}
\frac{
\begin{ytableau}
*(boxS) \\
*(boxS)  & *(boxS)  \\
*(boxU) \sU & *(boxU) \sU & *(boxU) \sU & *(boxG)  \\
\end{ytableau}\quad \begin{ytableau}
*(boxS) \\
*(boxS) &*(boxS) \\
\end{ytableau}
}{
\begin{ytableau}
*(boxS) \\
*(boxS) &*(boxS) \\
*(boxS)  & *(boxS)  &*(boxS)   \\
*(boxU) \sU & *(boxU) \sU & *(boxU) \sU & *(boxG)  \\
\end{ytableau}} \times  \jackjLR_{21,21}^{321} .\\
\end{equation*}
\end{example}


If $c_{\mu\nu}^{\lambda} = 1$, then the Strong Stanley conjecture \eqref{conj:strongstanley2} and the Jack Windowing conjecture \eqref{conj:jackwindowing} combine into the following statement.

\begin{conjecture}[Simple Windowing]\label{conj:simplewindowing}
Suppose $c_{\mu\nu}^{\lambda} = 1$ and let $\lambda/\mu \subset R$ be a window. By Lemma \eqref{lemma:schurlrwindow} we also have $c_{\mu_R\nu}^{\lambda_R}=1$. 
Assume the Strong Stanley conjecture \eqref{conj:strongstanley} holds. There exists a one-term rule $\JackjLR_{\mu_R,\nu}^{\lambda_R} = \bmD$ (i.e.  $[\JackjLR_{\mu_R,\nu}^{\lambda_R}] = \jackjLR_{\mu_R,\nu}^{\lambda_R}$), such that 
\[ \JackjLR_{\mu,\nu}^{\lambda} = \windfact_{R} \circ \left( \JackjLR_{\mu_R,\nu}^{\lambda_R} \right)_R, \]
is a one-term rule for $\jackjLR_{\mu,\nu}^{\lambda}.$ 
That is, this solves $[\JackjLR_{\mu,\nu}^{\lambda}] = \jackjLR_{\mu,\nu}^{\lambda}$. 

In other words, the Strong Stanley conjecture holds for $\jackjLR_{\mu,\nu}^{\lambda}$ (the evaluation of a single Stanley diagram with equal number of upper and lower hooks).
\end{conjecture}
The subtlety of this conjecture is in the non-uniqueness of the Stanley Diagrams.
\begin{example}
Consider the following two LR coefficients
\begin{equation}\label{eq:lrcalc1}
g_{2,21}^{32} = \frac{2\alpha(2+\alpha)}{2(1+\alpha)^2(1+2\alpha)}, \qquad \jackjLR_{22,21}^{322}=\frac{2\alpha^2(2+\alpha)}{(1+\alpha)^3(3+2\alpha)}.
\end{equation}
As these are both $c=1$ cases (and related through windowing), the strong Stanley conjecture predicts the existence of weight-0 Stanley diagrams $\JackjLR_{22,21}^{322}, \JackjLR_{2,21}^{32}$ that evaluate to these coefficients.
According to conjecture \ref{conj:simplewindowing}, these diagrams can be chosen in a way that they are compatible via windowing, i.e.
\[ \JackjLR_{22,21}^{322} = \frac{ \begin{ytableau}
*(boxU)\sU   &  *(boxU)\sU \\
*(boxS)    &  *(boxS) 
\end{ytableau}\,\,\begin{ytableau}
*(boxS)   \\
*(boxS)   &  *(boxS) 
\end{ytableau} }{ \begin{ytableau}
*(boxS)    & *(boxS)  \\
*(boxU)\sU   &  *(boxU)\sU \\
*(boxS)  & *(boxS)  & *(boxS) 
\end{ytableau} } \circ (\JackjLR_{2,21}^{32})_R\]
The first expression of \ref{eq:lrcalc1} forces the following hook choices for the diagram $\JackjLR_{2,21}^{32}$,
\[ \frac{ \begin{ytableau}
*(boxU)\sU   &  *(boxG)
\end{ytableau}\,\,\begin{ytableau}
*(boxG)  \\
*(boxL)\sL   &  *(boxG) 
\end{ytableau} }{ \begin{ytableau}
*(boxL)\sL   & *(boxG)  \\
*(boxL)\sL  & *(boxU)\sU & *(boxG) 
\end{ytableau} }, \]
where the unassigned boxes cannot be determined from \ref{eq:lrcalc1}, however they must contribute a factor of $1$. After de-windowing this partial assignment, the resulting diagram is.
\[ \frac{ \begin{ytableau}
*(boxU)\sU   &  *(boxU)\sU \\
*(boxU)\sU   &  *(boxG)  *
\end{ytableau}\,\,\begin{ytableau}
*(boxG)  \\
*(boxL)\sL   &  *(boxG) 
\end{ytableau} }{ \begin{ytableau}
*(boxL)\sL   & *(boxG)  \\
*(boxU)\sU   &  *(boxU)\sU \\
*(boxL)\sL  & *(boxU)\sU & *(boxG) 
\end{ytableau} } \]
This gives us the contribution
\[  \frac{\alpha^2(2+\alpha)}{(3+2\alpha)(1+\alpha)^3} \]

Thus the box with the $(*)$ must be a lower hook, to give the missing factor of 2, which results in 
\[ \frac{ \begin{ytableau}
*(boxU)\sU   & *(boxL)\sL
\end{ytableau}\,\,\begin{ytableau}
*(boxG)  \\
*(boxL)\sL   &  *(boxG) 
\end{ytableau} }{ \begin{ytableau}
*(boxL)\sL   & *(boxG)  \\
*(boxL)\sL  & *(boxU)\sU & *(boxG) 
\end{ytableau} }, \]

Importantly, note that this assignment differs to what would have been determined by applying the Pieri rule (although the $\nu,\mu$ factors are in the wrong order).

\end{example}

\begin{example}
We show that the Pieri rule~\eqref{thm:pieri} is a special case of the Simple Windowing Conjecture \eqref{conj:simplewindowing}. We consider the earlier example \eqref{ex:pieri}, however now we note that we can cancel several hooks in the fraction that appears within, to recover a simpler picture:
\begin{equation*}
\JackjLR_{75411,3}^{76431} =
\ytableausetup{boxsize=0.8em}
\frac{
\begin{ytableau}
 *(boxG)\\
*(boxL) \sL \\
 *(boxG)& *(boxU) \sU  & *(boxU) \sU  & *(boxG) \\
*(boxL) \sL & *(boxU) \sU  & *(boxU) \sU  & *(boxL) \sL  &*(boxL) \sL \\
 *(boxG) & *(boxU) \sU  &*(boxU) \sU  & *(boxG) & *(boxG)  & *(boxU) \sU & *(boxG)
\end{ytableau}
\quad
\begin{ytableau}
*(boxU) \sU & *(boxU) \sU & *(boxU) \sU
\end{ytableau}}{
\begin{ytableau}
 *(boxG)\\
*(boxL) \sL & *(boxU) \sU  & *(boxU) \sU \\
 *(boxG)& *(boxU) \sU  & *(boxU) \sU  & *(boxG) \\
*(boxL) \sL  & *(boxU) \sU  & *(boxU) \sU  & *(boxL) \sL &*(boxL) \sL  & *(boxU) \sU \\
 *(boxG)& *(boxU) \sU  &*(boxU) \sU  & *(boxG)  & *(boxG)  & *(boxU) \sU & *(boxG) 
\end{ytableau}} \\
\end{equation*}
We split this up according to the window $R = \meetjoin(\lambda/\mu)$:

\begin{equation*}
\ytableausetup{boxsize=0.8em}
\frac{
\begin{ytableau}
 *(boxG)\\
*(boxL) \sL \\
 *(boxG)& *(boxU) \sU  & *(boxU) \sU  & *(boxG) \\
*(boxL) \sL & *(boxS)   & *(boxS)   & *(boxL) \sL  &*(boxL) \sL \\
 *(boxG) & *(boxU) \sU  &*(boxU) \sU  & *(boxG) & *(boxG)  & *(boxU) \sU & *(boxG) 
\end{ytableau}
\quad
\begin{ytableau}
*(boxS)  & *(boxS) & *(boxS) 
\end{ytableau}}{
\begin{ytableau}
 *(boxG)\\
*(boxL) \sL & *(boxS)  & *(boxS) \\
 *(boxG)& *(boxU) \sU  & *(boxU) \sU  & *(boxG) \\
*(boxL) \sL  & *(boxS)  & *(boxS)   & *(boxL) \sL &*(boxL) \sL  & *(boxS)  \\
 *(boxG)& *(boxU) \sU  &*(boxU) \sU  & *(boxG)  & *(boxG)  & *(boxU) \sU & *(boxG) 
\end{ytableau}} \\
\end{equation*}
We see that this factor is precisely $\windfact_{R}$, as predicted by the conjecture.
The remaining windowed (i.e. beige) boxes are filled in with the de-windowed basic Pieri rule,
\begin{equation}
\JackjLR_{2,3}^{32} = \ytableausetup{boxsize=0.8em}
\frac{
\begin{ytableau}
*(boxU) \sU & *(boxU) \sU 
\end{ytableau}
\quad
\begin{ytableau}
*(boxU) \sU & *(boxU) \sU & *(boxU) \sU
\end{ytableau}}{
\begin{ytableau}
*(boxU) \sU & *(boxU) \sU   \\
*(boxU) \sU & *(boxU) \sU & *(boxU) \sU 
\end{ytableau}}. \\
\end{equation}
\end{example}

\begin{example}\label{ex:stanelystableaureworked}
Consider the earlier Maximal Filling tableau example \eqref{ex:stanleytableau} of $\lambda = 766654211$, $\mu = 7553322$, $\nu=542$. We will show this is also an example of the Jack Windowing conjecture~\eqref{conj:jackwindowing}. As in the previous example, we cancel out hooks that appear equally in the numerator and denominator and extract the minimal window $R=\meetjoin(\lambda/\mu)$
\begin{equation*}
\ytableausetup{boxsize=0.8em}
\frac{
\begin{ytableau}
*(boxU) & *(boxL)   \\
*(boxU) & *(boxL)   \\
*(boxU) & *(boxL)  & *(boxL)  \\
*(boxL) & *(boxL)  & *(boxL)   \\
*(boxU) & *(boxL)  & *(boxL)  & *(boxU) & *(boxU)  \\
*(boxL) & *(boxL)  & *(boxL)  & *(boxU) & *(boxL)  \\
*(boxU) & *(boxL)  & *(boxU)  & *(boxU) & *(boxU) & *(boxU) & *(boxL) \\
\end{ytableau} \quad
\begin{ytableau}
*(boxU)  &*(boxU)   \\
*(boxU)  &*(boxU)  &*(boxU)  &*(boxU)     \\
*(boxU)  &*(boxU)  &*(boxU)  &*(boxU)  &*(boxU)  
\end{ytableau}}{
\begin{ytableau}
*(boxU)  \\
*(boxU)  \\
*(boxU) & *(boxL)   \\
*(boxL) & *(boxL)  & *(boxL)  & *(boxU) \\
*(boxU) & *(boxL)  & *(boxL)  & *(boxU) & *(boxU)\\
*(boxL) & *(boxL)  & *(boxL)  & *(boxU) & *(boxL) & *(boxU) \\
*(boxU) & *(boxL)  & *(boxL)  & *(boxU) & *(boxU) & *(boxU) \\
*(boxU) & *(boxL)  & *(boxU)  & *(boxU) & *(boxU) & *(boxU) \\
*(boxU) & *(boxL)  & *(boxU)  & *(boxU) & *(boxU) & *(boxU) & *(boxL)\\
\end{ytableau}}
\,\,\to\,\,
\frac{
\begin{ytableau}
*(boxU) & *(boxG)    \\
*(boxS) & *(boxL)   \\
*(boxS) & *(boxL)  & *(boxS)  \\
*(boxS) & *(boxL)  & *(boxS)   \\
*(boxS) & *(boxL)  & *(boxS)  & *(boxS) & *(boxS)  \\
*(boxS) & *(boxL)  & *(boxS)  & *(boxS) & *(boxS)  \\
*(boxU) & *(boxG)   & *(boxU)  & *(boxU) & *(boxU) & *(boxU) & *(boxG) \\
\end{ytableau} \quad
\begin{ytableau}
*(boxS)  &*(boxS)   \\
*(boxS)  &*(boxS)  &*(boxS)  &*(boxS)     \\
*(boxS)  &*(boxS)  &*(boxS)  &*(boxS)  &*(boxS)  
\end{ytableau}}{
\begin{ytableau}
*(boxS)  \\
*(boxS)  \\
*(boxU) & *(boxG)  \\
*(boxS) & *(boxL)  & *(boxS)  & *(boxS) \\
*(boxS) & *(boxL)  & *(boxS)  & *(boxS) & *(boxS)\\
*(boxS) & *(boxL)  & *(boxS)  & *(boxS) & *(boxS) & *(boxS) \\
*(boxS) & *(boxL)  & *(boxS)  & *(boxS) & *(boxS) & *(boxS) \\
*(boxS) & *(boxL)  & *(boxS)  & *(boxS) & *(boxS) & *(boxS) \\
*(boxU) & *(boxG)  & *(boxU)  & *(boxU) & *(boxU) & *(boxU) & *(boxG) \\
\end{ytableau}}  \\
\end{equation*}
We see the right hand side agrees with the $\windfact_{R}$ factor.
The beige windowed region is then filled in with the following maximal filling rule
\begin{equation}\label{fig:stanelyfactor}
\JackjLR_{44221,542}^{5554311} =
\ytableausetup{boxsize=0.8em}
\frac{
\begin{ytableau}
*(boxU)   \\
*(boxU)  & *(boxL)  \\
*(boxL)  & *(boxL)   \\
*(boxU)   & *(boxL)  & *(boxU) & *(boxU)  \\
*(boxL)   & *(boxL)  & *(boxU) & *(boxL)  \\
\end{ytableau} \quad
\begin{ytableau}
*(boxU)  &*(boxU)   \\
*(boxU)  &*(boxU)  &*(boxU)  &*(boxU)     \\
*(boxU)  &*(boxU)  &*(boxU)  &*(boxU)  &*(boxU)  
\end{ytableau}}{
\begin{ytableau}
*(boxU)  \\
*(boxU)  \\
*(boxL)   & *(boxL)  & *(boxU) \\
*(boxU)  & *(boxL)  & *(boxU) & *(boxU)\\
*(boxL)  & *(boxL)  & *(boxU) & *(boxL) & *(boxU) \\
*(boxU)   & *(boxL)  & *(boxU) & *(boxU) & *(boxU) \\
*(boxU)   & *(boxU)  & *(boxU) & *(boxU) & *(boxU) \\
\end{ytableau}}, \\
\end{equation}
where the maximal filling is now
\[ T_R =\ytableausetup{boxsize=0.8em} \begin{ytableau}
2  \\
1 \\
    & 1 & 3 \\
  &  & 2& 2\\
  &    & 1 & 1 & 3 \\
  &   & &  & 2 \\
  &    &  &  & 1 \\
\end{ytableau}.\]
Indeed, this rule could also be determined by a sequence of windowing:
\begin{equation}
\ytableausetup{boxsize=0.8em}
\frac{
\begin{ytableau}
 *(boxU)  \\
\end{ytableau} \,\,
\begin{ytableau}
*(boxU)  &*(boxU)   \\
\end{ytableau}}{
\begin{ytableau}
 *(boxU) \\
 *(boxU) & *(boxU) \\
\end{ytableau}} \to 
\frac{
\begin{ytableau}
*(boxL) \mk   \\
*(boxL)  \mk & *(boxU) & *(boxL) \mk  \\
\end{ytableau} \,\,
\begin{ytableau}
*(boxU)  &*(boxU)   \\
\end{ytableau}}{
\begin{ytableau}
*(boxL) \mk & *(boxU) \\
*(boxL) \mk & *(boxU) & *(boxL) \mk  & *(boxU) \\
\end{ytableau}} \to 
\frac{
\begin{ytableau}
*(boxU) \mk  \\
*(boxU) \mk   \\
*(boxL)    \\
*(boxL)   & *(boxU) & *(boxL)  \\
\end{ytableau} \,\,
\begin{ytableau}
*(boxU)  &*(boxU)   \\
\end{ytableau}}{
\begin{ytableau}
*(boxU) \mk \\
*(boxU) \mk \\
*(boxL)  & *(boxU) \\
*(boxL)  & *(boxU) & *(boxL) & *(boxU) \\
\end{ytableau}} \to 
\frac{
\begin{ytableau}
*(boxU)   \\
*(boxU)    \\
*(boxL)    \\
*(boxU)   \mk & *(boxU)  \mk & *(boxU)  \mk \\
*(boxL)   & *(boxU) & *(boxL)  \\
\end{ytableau} \,\,
\begin{ytableau}
*(boxU)  &*(boxU)   \\
\end{ytableau}}{
\begin{ytableau}
*(boxU)  \\
*(boxU)  \\
*(boxL)  & *(boxU) \\
*(boxU) \mk & *(boxU)  \mk& *(boxU)  \mk\\
*(boxL)  & *(boxU) & *(boxL) & *(boxU) \\
\end{ytableau}}  \to
\end{equation} 

\begin{equation}
\ytableausetup{boxsize=0.8em}
\frac{
\begin{ytableau}
*(boxU)   \\
*(boxU)    \\
*(boxL)    \\
*(boxU)   & *(boxU) & *(boxU)  \\
*(boxL)   & *(boxU) & *(boxL)  \\
\end{ytableau} \,\,
\begin{ytableau}
*(boxU)  &*(boxU)   \\
*(boxU)  \mk &*(boxU) \mk &*(boxU) \mk &*(boxU)   \mk
\end{ytableau}}{
\begin{ytableau}
*(boxU)  \\
*(boxU)  \\
*(boxL)  & *(boxU) \\
*(boxU)  & *(boxU) & *(boxU)\\
*(boxL)  & *(boxU) & *(boxL) & *(boxU) \\
*(boxU)  \mk& *(boxU) \mk& *(boxU) \mk& *(boxU) \mk\\
\end{ytableau}} \to
\frac{
\begin{ytableau}
*(boxU)   \\
*(boxU)  & *(boxL) \mk \\
*(boxL)  & *(boxL) \mk  \\
*(boxU)   & *(boxL) \mk & *(boxU) & *(boxU)  \\
*(boxL)   & *(boxL) \mk & *(boxU) & *(boxL)  \\
\end{ytableau} \,\,
\begin{ytableau}
*(boxU)  &*(boxU)   \\
*(boxU)  &*(boxU)  &*(boxU)  &*(boxU)   
\end{ytableau}}{
\begin{ytableau}
*(boxU)  \\
*(boxU)  \\
*(boxL)   & *(boxL) \mk & *(boxU) \\
*(boxU)  & *(boxL) \mk & *(boxU) & *(boxU)\\
*(boxL)  & *(boxL) \mk & *(boxU) & *(boxL) & *(boxU) \\
*(boxU)   & *(boxL) \mk & *(boxU) & *(boxU) & *(boxU) \\
\end{ytableau}} \to
\frac{
\begin{ytableau}
*(boxU)   \\
*(boxU)  & *(boxL)  \\
*(boxL)  & *(boxL)   \\
*(boxU)   & *(boxL)  & *(boxU) & *(boxU)  \\
*(boxL)   & *(boxL)  & *(boxU) & *(boxL)  \\
\end{ytableau} \,\,
\begin{ytableau}
*(boxU)  &*(boxU)   \\
*(boxU)  &*(boxU)  &*(boxU)  &*(boxU)     \\
*(boxU) \mk &*(boxU)  \mk&*(boxU)  \mk&*(boxU) \mk &*(boxU)  \mk
\end{ytableau}}{
\begin{ytableau}
*(boxU)  \\
*(boxU)  \\
*(boxL)   & *(boxL)  & *(boxU) \\
*(boxU)  & *(boxL)  & *(boxU) & *(boxU)\\
*(boxL)  & *(boxL)  & *(boxU) & *(boxL) & *(boxU) \\
*(boxU)   & *(boxL)  & *(boxU) & *(boxU) & *(boxU) \\
*(boxU) \mk  & *(boxU) \mk & *(boxU) \mk& *(boxU) \mk& *(boxU) \mk\\
\end{ytableau}}.
\end{equation}

\end{example}

We can show the example of \eqref{ex:stanelystableaureworked} holds in general.

\begin{theorem} 
The Simple Windowing conjecture~\eqref{conj:simplewindowing} holds in the Maximal Filling Tableau case.
\end{theorem}

\begin{proof}
Firstly, we show that outside $R= \meetjoin(\lambda/\mu)$ the assignment rule~\eqref{eq:maxfilling} yields a factor of $\windfact_{R}$. We note that $R$ can be described as the union of all columns satisfying $c_j=0$ with all rows satisfying $r_i = 0$. For each box $b=(j,i)\in \lambda$ in a column with $c_j=0$, the formula~\eqref{eq:maxfillingrule} assigns $A(T,b)=\sL$. For such boxes we have $b_*=b$, and thus $b \in \mu$ is also assigned $\sL$. Thus each such column agrees with $\windfact_{R}$. 

Next we look at boxes $b=(j,i)$ in rows with $r_i=0$. We need to show that $A(T,b)= A(T,b_*)$, and that these are equal to $\sU$ if $c_j>0$. If $c_j$ is 0, then we see that $A(T,b)= A(T,b_*)=\sL$. Next, if $c_j>0$, we have $A(T,b)=\sU$. So we just need to check that $A(T,b_*)=\sU$. If $r_i=0$, then $r_{i+c_j}$ is at most $c_j$, and thus $A(T,b_*)=\sU$ and we are done. Thus, outside of R we assign hooks according to $\windfact_{R}$.

Next, we need to show that the assignment of hooks inside $R$ agrees with the same rule applied to $\lambda_R, \mu_R$. First we note that the maximal filling LR tableau agrees with its windowed version $T_R$. Hence the vectors of windowed row-column maximums $c^R, r^R$ , for the windowed LR tableau $T_R$ agree with the un-windowed $c,r$ however will all zeros removed. Thus, for each box $b' \coloneqq (j',i')$ in the windowed partitions $\lambda_R$, corresponding to the unwindowed box $b=(j,i)$ in $R$, we have the equivalent conditions $c^R_{j'} \leq r_{i'}^R \Leftrightarrow c_j \leq r_i$. Hence the windowed choice of hooks inside $\lambda_R$ agrees with the un-windowed choices inside $R \cap \lambda$, i.e. $A(T_R,b')=A(T,b)$.


So, to complete the proof we have to show that the hook choices all boxes $b'=(j',i') \in \mu_R$ also agree with those for their un-windowed counterparts $b=(j,i) \in \mu$, that is, we need to show that 
\begin{equation}\label{eq:Dmu}
A(T_R,b')=A(T,b),
\end{equation}
or equivalently 
$A(T_R,(b')_*) = A(T,b_*)$, where as before $b_* \coloneqq (j,i+c_j)$.

The complication is that we may have removed certain rows from $\mu$ that have $r_i=0$ (i.e, those in $R$), and so we have to check that these hook choices still agree after taking into consideration the vertical shift by $c_j$. In other words, it is \emph{not} the case that $(b')_*=(b_*)' $ for all $b$. If indeed this was so, equation~\eqref{eq:Dmu} would follow immediately, since then we would have 
$A(T_R,(b')_*) =A(T_R,(b_*)') = A(T,b_*)$.
This complication is demonstrated in the following example of $\mu=44441, \nu = 431, \lambda = 5554321$.
\ytableausetup{boxsize=1.0em}
\[ T := \begin{ytableau}
 2\\
1& 2\\
& 1 &1 \\
&  & & \\
&  & & & 3 \\
&  & & & 2 \\
&  & &  &1 \\
\end{ytableau}, \quad T_R := \begin{ytableau}
 2\\
1& 2\\
& 1 &1 \\
&  & &  3 \\
&  & &  2 \\
&  & &  1 \\
\end{ytableau} \]
 \begin{equation*}
\ytableausetup{boxsize=1.0em}
\frac{
\begin{ytableau}
*(boxU) a\\
*(boxU) b& *(boxU) &*(boxU) &*(boxL)\\
*(boxU) c & *(boxU) &*(boxU)& *(boxL)\\
*(boxU) d& *(boxU) &*(boxL) &*(boxL) & \none \\
*(boxL) e& *(boxL) &*(boxL) &*(boxL) & \none \\
\end{ytableau}
}{
\begin{ytableau}
*(boxU)a'\\
*(boxU) b_*&*(boxU)  \\
*(boxU) c_*&  *(boxU) & *(boxU)  \\
*(boxU) d_*& *(boxU) &*(boxU) & *(boxL) \\
*(boxL) e_*& *(boxL) &*(boxL) &*(boxL) & *(boxU)  \\
*(boxU)&*(boxU)  & *(boxL)& *(boxL) &*(boxU)   \\
*(boxU)&*(boxU)  & *(boxU)& *(boxL) &*(boxU)   \\
\end{ytableau}} \leftrightarrow
\frac{
\begin{ytableau}
*(boxU)a'\\
*(boxU)c'&*(boxU)  &*(boxU)  \\
*(boxU) d'& *(boxU) & *(boxL)& \none \\
*(boxL) e'& *(boxL) & *(boxL)& \none \\
\end{ytableau} }{
\begin{ytableau}
\none \\
*(boxU)a'_*  \\
*(boxU)f & *(boxU) \\
*(boxU)g& *(boxU)  &*(boxU) \\
*(boxL)e'_*&*(boxL) &*(boxL) &  *(boxU) \\
*(boxU)& *(boxU) & *(boxL)& *(boxU)  \\
*(boxU)& *(boxU) & *(boxU) & *(boxU)  \\
\end{ytableau}}. \\
\end{equation*}
In this example, we find the following box correspondences,
\[ (a')_*=(a_*)', \quad  f= (b_*)'=(c')^*, \quad g=(c_*)' = (d')_* \quad (e')_*=(e_*)'. \]
In $\mu$ the row containing $b$ is removed, in $\lambda$ the row containing $d_*$ is removed. Here we see that the two boxes $c,d$ in $\mu$, ie, those that are in rows at up to $c_j$ below the row that is being removed (the one containing $b$), need to be shifted up by 1 row inside $\lambda$ for the correspondence to work out. 

We will show that indeed the hook choices are consistent when one removes only a single row, and then by iteration the result will hold for all $R$. So let $R$ be the complement of a single row, $\tau$. All the considerations below are for boxes in a column with $c_j >0$.

We have shown that for boxes $x \in \mu$ at most $c_j$ rows beneath the row $\tau$ in $\mu$, which we say are \emph{close} to $\tau$, we have
\[ (x_* + (0,1))' = (x')_*, \]
whereas for all other boxes one can easily show $(x_*)'=(x')_*$.
We have to check now that the rule~\eqref{eq:maxfillingrule} assigns consistent hooks for $x$ close to $\tau$ across the four boxes $x, x', (x')_*$ and $x_* + (0,1)$.



For any $x$ that is between $1$ and $c(x)$ rows below $\tau$ in $\mu$, $x_* = x+(0,c(x))$ is at most $c(x)-1$ rows \emph{above} the row $\tau$ in $\lambda$. Then let $y_*=x_*+(0,1)$ be the box that satisfies $(y_*)' = (x')_*$. We then know that $y_*$ is at most $c(x)$ rows above the row $\tau$ in $\lambda$.

Next, consider $z_* \in \lambda$ in the row $\tau$ in $\lambda$, for which we have $c(z_*)>0$ and $r(z_*) = 0$. We will show that $A(T,z_*) = \sU = A(T,z_*+(0,k) )$ for all $k \leq c(z_*)$. In particular this applies to $y_*$ from the previous paragraph. We recall that the maximal filling condition forces the the entries of the LR tableau to be increasing incrementally as we move up columns, and so because we know that $r(z_*)=0$ we must have $r(z_* + (0,k))\leq k$. Thus, for $k \leq c(z_*)$, we find $r(z_* + (0,k)) \leq c(z_*) = c(z_* + (0,k)) > 0$, and so $A(T,z_*+(0,k) ) = \sU$.

Thus for $x$ close to $\tau$, we have shown the following agreement of $\mu$-hook assignments:
 \begin{eqnarray*}
A(T,x_*) &=& A(T,x_*+(0,1)) \\&=& A(T_R,(y_*)') \\&=& A(T_R,(x')_*)
 \end{eqnarray*}
For $x$ not close to $\tau$, we have $(x')_* = (x_*)'$ and so the same agreement follows directly. Thus the theorem is proven.
\end{proof}

\subsection{Positivity Conjecture}

With the constructions of the previous sections, we can propose a sharpening of the Stanley Conjecture~\eqref{conj:stanley}.

\begin{conjecture}[Window Positivity]\label{conj:windowpositivity} Let $\mu \subset \lambda$ be such that $\lambda/\nu \subset R$, for a window $R$. In terms of the positivity of Conjecture~\eqref{conj:stanley}, the factor $F_{R}$ accounts for all boxes outside $R$ in the following precise sense:
\begin{equation} \jackjLR_{\mu\nu}^\lambda \cdot \left( F_{\mu,R}^\lambda\right)^{-1} \cdot j_{\lambda,R} \in \setZ_{\geq0}[\alpha], \end{equation}
where $j_{\lambda,R} \coloneqq \prod_{b \in \lambda \cap R} \uh{\lambda}{b} \lh{\lambda}{b} $. 
\end{conjecture}

Note that $F_{R}$ is a rational function in $\alpha$ with numerator and denominator each a positive polynomial of equal degrees, and $j_{\lambda,R}$ is a positive polynomial of degree $2|\lambda \cap R|$, and so conjecture~\eqref{conj:windowpositivity} claims a stronger positivity than that of Conjecture~\eqref{conj:stanley}.

\begin{example} To demonstrate conjecture \eqref{conj:windowpositivity}, we take $\mu = 654321$, $\nu=211$, $\lambda=664432$. We find
\[ \jackjLR_{\mu,\nu}^{\lambda} = \frac{\alpha^2 (2+3\alpha) f(\alpha)}{2700 (1+\alpha)^{12} (3+2\alpha)(4+3\alpha)} \]
where $f(\alpha)$ is an irreducible degree 8 positive polynomial.
We take the minimal window $R$, 
\begin{equation*}
\ytableausetup{boxsize=0.7em}R \coloneqq
\begin{ytableau}
{}& *(boxS)   & *(boxS)  & *(boxS)& & *(boxS) \\
& *(boxS)   & *(boxS)  & *(boxS)& & *(boxS) \\
 & *(boxS)  & *(boxS) & *(boxS)& & *(boxS)\\
  &  &  &&  &    \\
& *(boxS)   &*(boxS)   & *(boxS) & & *(boxS) \\
  &   &  & & & \\
\end{ytableau}  \\
\end{equation*}
for which we find

\begin{equation*}
{\ytableausetup{boxsize=0.7em}
\windfact_{\mu,R}^\lambda = \frac{\begin{ytableau}
*(boxL)  \\
*(boxL)& *(boxS)   \\
*(boxL) & *(boxS)  & *(boxS) \\
 *(boxG)  & *(boxU) & *(boxU)  & *(boxU)  \\
*(boxL) & *(boxS)   &*(boxS)   & *(boxS) & *(boxL) \\
 *(boxG)  & *(boxU)  & *(boxU) & *(boxU)& *(boxG) & *(boxU)  \\
\end{ytableau}}{
\begin{ytableau}
*(boxL) & *(boxS)     \\
*(boxL)& *(boxS)   & *(boxS)   \\
*(boxL) & *(boxS)  & *(boxS) & *(boxS)\\
 *(boxG)  & *(boxU) & *(boxU)  & *(boxU)  \\
*(boxL) & *(boxS)   &*(boxS)   & *(boxS) & *(boxL) & *(boxS) \\
 *(boxG)  & *(boxU)  & *(boxU) & *(boxU)& *(boxG) & *(boxU)  \\
\end{ytableau}}
}, \quad F_{\mu,R}^\lambda = \frac{\hkl{1}{0}\hkl{2}{1}\hkl{3}{2}\hkl{5}{4}\hkl{2}{3}\hkl{1}{2}\hkl{0}{1}\hkl{4}{5}\hkl{3}{4}\hkl{2}{3}\hkl{1}{0}\hkl{0}{1} }{ \hkl{1}{1}\hkl{2}{2}\hkl{3}{3}\hkl{5}{5}\hkl{3}{3}\hkl{2}{2}\hkl{1}{1}\hkl{5}{5}\hkl{4}{4}\hkl{3}{3}\hkl{1}{1}\hkl{1}{1} } 
\end{equation*}
where $\hkl{a}{b} \coloneqq a+b\alpha$. We also have
\[ j_{\lambda,R} = (\!\hkl{0}{1}\hkl{1}{0})^4(\!\hkl{1}{2}\hkl{2}{1})^2\hkl{2}{3}\hkl{3}{2}\hkl{4}{5}\hkl{5}{4}\hkl{3}{4}\hkl{4}{3}\hkl{2}{3}\hkl{3}{2}. \]
We check the window positivity conjecture
\[ \jackjLR_{\mu,\nu}^{\lambda} \cdot \left(F_{\mu,R}^\lambda\right)^{-1} \cdot j_{\lambda,R} = 4\alpha^4(2+\alpha)(1+2\alpha)(2+3\alpha) f(\alpha) \in \setZ_{\geq0}[\alpha]. \]
Note that the Stanley conjecture \eqref{conj:stanley} gives the weaker positivity property
\[ \jackjLR_{\mu,\nu}^{\lambda}\cdot  j_{\lambda} = (\text{ degree 32 non-negative polynomial in } \alpha \, ) f(\alpha) \in \setZ_{\geq0}[\alpha]. \]
\end{example}

\begin{example}[P. Hanlon's example {\cite[p114]{Stanley:1989}}]\label{example:hanlonpositivity}
Take $\mu = 31$, $\nu=21$, $\lambda=421$. Here we have $c=2$, and the LR coefficient is given by
\[ \jackjLR_{31,21}^{421} = \frac{\alpha( 9+97\alpha+294\alpha^2 + 321\alpha^3 + 131\alpha^4+12\alpha^5 )}{3(1+\alpha)^3(2+\alpha)(1+2\alpha)^2(1+3\alpha)} \]
We take the minimal $R= \meetjoin(\lambda/\mu)$, which gives
\[ F_{R} = \frac{
\ytableausetup{boxsize=0.8em}
\begin{ytableau}
*(boxS)  \\
*(boxS) & *(boxS) & *(boxL) \sL    \\
\end{ytableau}
}{
\begin{ytableau}
*(boxS) \\
*(boxS) & *(boxS)\\
*(boxS) & *(boxS) & *(boxL) \sL & *(boxS)   \\
\end{ytableau}} = \frac{1}{1+\alpha}.
\]
To verify conjecture \eqref{conj:windowpositivity}, we check
\[  \jackjLR_{31,21}^{421} \cdot F_{R}^{-1}\cdot j_{\lambda,R}=  4\alpha^4( 9+97\alpha+294\alpha^2 + 321\alpha^3 + 131\alpha^4+12\alpha^5 ), \]
with
\begin{eqnarray*} 
j_{\lambda,R} &=& \hkl{0}{1}^3\hkl{1}{2}\hkl{2}{1}\hkl{1}{3}\hkl{2}{2}\hkl{2}{3}\hkl{3}{3} \\
&=& 12 \,\alpha^3(1+2\alpha)^2(2+\alpha)(1+3\alpha)(1+\alpha)^2.  
\end{eqnarray*}
We revisit this example in \eqref{ex:hanloncalc}.
\end{example}

\section{The action of $J_{21}$}\label{section:J12action}
In \cite{Stanley:1989}, the action of $J_{21}$ on the ring of symmetric funcitons (in the Jack basis) was partially determined.

\begin{theorem}[Stanley {\cite[p114]{Stanley:1989}}]
If $c_{\mu,21}^{\lambda}=1$, then the strong Stanley conjecture holds for $\jackjLR_{\mu,21}^{\lambda}$.
\end{theorem}
No explicit Stanley Diagram is provided in that work, however is it deducible from a given formula. Using the windowing conjecture, we can address the above problem in the following way

\begin{corollary}
Assume the Jack Windowing conjecture holds. If $c_{\mu,21}^{\lambda}=1$, then the problem of computing $\jackjLR_{\mu,21}^{\lambda}$ is reducible via windowing of $\lambda$ and $\mu$ to one of the following 3 cases (up to transposition).
\begin{itemize}
\item $\lambda/\mu \cong \{21\}$. By the locality property \eqref{thm:locality}, this reduces to $\JackjLR_{\emptyset,21}^{21}$.
\item $\lambda/\mu \cong \{1,1\} \cup \{1\}$. This reduces to that of $\JackjLR_{1,21}^{211}$. 
\item $\lambda/\mu \cong \{2\} \cup \{1\}$. This reduces to that of $\JackjLR_{2,21}^{32}$.  
\end{itemize}
If $c_{\mu,21}^{\lambda}=2$, then the problem is reducible via windowing of $\lambda$ and $\mu$ to
\begin{itemize}
\item
$\lambda/\mu \cong \{1\} \cup \{1\} \cup \{1\}$. In this case, the problem reduces to $\JackjLR_{21,21}^{321}$.
\end{itemize}
\end{corollary}

For example, in the case where $\lambda/\mu \cong \{2\} \cup \{1\}$, we find a solution of the form
\begin{equation}\label{eq:21eq}
\ytableausetup{boxsize=0.8em}
 \JackjLR_{\mu,21}^{\lambda} = \frac{\begin{ytableau}
\none  &   \\
\none  &  *(boxL) \\
\none  & & *(boxU)  & *(boxU) \\
\none  & & *(boxU)  & *(boxU)  &  \\
\none  & *(boxL) & *(boxS)  & *(boxS)  & *(boxL) \\
\none  & & *(boxU)  & *(boxU)  &  & *(boxU) &    \\
\none[\mu] &  \none & \none  & \none & \none
\end{ytableau} \,\, \begin{ytableau}
\none\\
\none\\
\none\\
\none \\
\none & *(boxS)  \\
\none & *(boxS)   & *(boxS)   \\
\none & \none & \none 
\end{ytableau}
}{\begin{ytableau}
\none  &   \\
\none  &  *(boxL) & *(boxS)  & *(boxS) \\
\none  & & *(boxU)  & *(boxU) \\
\none  & & *(boxU)  & *(boxU)  &  \\
\none  & *(boxL) & *(boxS)  & *(boxS)  & *(boxL)& *(boxS) \\
\none  & & *(boxU)  & *(boxU)  &  & *(boxU) &    \\
\none[\lambda] &  \none & \none  & \none & \none
\end{ytableau}} \circ \JackjLR_{2,21}^{32}
\ytableausetup{boxsize=0.8em}
\end{equation}

For the rest of of this section we turn our attention to the $c=2$ case of $\lambda/\mu \cong \{1\} \cup \{1\} \cup \{1\}$.

\subsection{ The $\lambda/\mu \cong \{1\} \cup \{1\} \cup \{1\}$ case }


An example of this is given by $\lambda = 98766532$, $\mu = 97765522$. The Windowing conjecture predicts the existence of a solution of the form

\begin{equation}\label{eq:c2cases1}
\ytableausetup{boxsize=0.8em}
 \JackjLR_{\mu,21}^{\lambda} = \frac{\begin{ytableau}
\none &   	     &   	     \\
\none & *(boxL)  & *(boxL) & \none[*] \\
\none & 		 & 		   & *(boxU)&			 & \\
\none & *(boxL)  & *(boxL) & *(boxS)&  *(boxL) &  *(boxL)  & \none[*] \\
\none & 		 & 		   & *(boxU)&			 & 			 & *(boxU) \\
\none & 		 & 		   & *(boxU)&			 & 			 & *(boxU)  &  \\
\none & *(boxL)  & *(boxL) & *(boxS)&  *(boxL) &  *(boxL)  & *(boxS)  & *(boxL)  & \none[*] \\
\none &   		 &   	   & *(boxU)&			 & 			 & *(boxU)  &  & *(boxU) &    \\
\none[\mu] &  \none & \none  & \none & \none
\end{ytableau} \,\, \begin{ytableau}
\none\\
\none\\
\none\\
\none \\
\none \\
\none \\
\none & *(boxS)  \\
\none & *(boxS) & *(boxS)   \\
\none & \none & \none 
\end{ytableau}
}{\begin{ytableau}
\none &   	     &   \\
\none & *(boxL)  & *(boxL) & *(boxS) *  \\
\none & 		 & 		 & *(boxU)&			 & \\
\none & *(boxL)  & *(boxL) & *(boxS)&  *(boxL) &  *(boxL)  & *(boxS) * \\
\none & 		 & 		 & *(boxU)&			 & 			 & *(boxU) \\
\none & 		 & 		 & *(boxU)&			 & 			 & *(boxU)  &  \\
\none & *(boxL)  & *(boxL) & *(boxS)&  *(boxL) &  *(boxL)  & *(boxS)  & *(boxL)& *(boxS) * \\
\none &   		 &   		 & *(boxU)&			 & 			 & *(boxU)  &  		 & *(boxU) &    \\
\none[\lambda] &  \none & \none  & \none & \none
\end{ytableau}} \circ \JackjLR_{21,21}^{321}
\ytableausetup{boxsize=0.8em}
\end{equation}

The most generic case can be described by six integers $(m_0,m_1,m_2,n_0,n_1,n_2)$ and 4 auxiliary partitions $\sigma_i, i=1,\ldots4$, which have a window $R$ described by the factor
\begin{equation}\label{eq:4paramwindowgeneral}
\ytableausetup{boxsize=1.0em}
 \windfact_{\mu,21}^{\lambda} = \frac{\begin{ytableau}
\none  & \sigma_0  \\
\none  & *(boxL)  \\
\none[m_2\,]  & &*(boxU) & \sigma_1 \\
\none         & *(boxL)&*(boxS) a_1 &*(boxL) \\
\none[m_1\,]  & &*(boxU)  & *(boxG)  & *(boxU) & \sigma_2 \\
\none         & *(boxL)&*(boxS) a_2 &*(boxL)  &*(boxS)  a_3& *(boxL)   \\
\none[m_0\,]  & & *(boxU) &  & *(boxU)& &*(boxU) & \sigma_4 \\
\none[\mu]    & \none[n_0] &  \none & \none[n_2]  & \none & \none[n_1]
\end{ytableau} \,\, \begin{ytableau}
\none\\
\none \\
\none & *(boxS) b_1  \\
\none & *(boxS) b_2  & *(boxS) b_3  \\
\none[\nu] & \none & \none 
\end{ytableau}
}{\begin{ytableau}
\none 		 &	\sigma_0 \\
\none 		 &	*(boxL)& *(boxS) c_1  \\
\none[m_2\,] &  &*(boxU)& \sigma_1\\
\none 		 &	*(boxL)& *(boxS) c_2 &*(boxL)  & *(boxS) c_3  \\
\none[m_1\,] & 	&*(boxU) & *(boxG) &*(boxU) & \sigma_2  \\
\none 		 &	*(boxL)& *(boxS) c_4 &*(boxL)  &*(boxS)  c_5   & *(boxL)  & *(boxS) c_6 \\
\none[m_0\,] &	& *(boxU)   &   &  *(boxU)     &   & *(boxU) & \sigma_4 \\
\none[\lambda] &\none[n_0]	& \none  & \none[n_2] & \none & \none[n_1] & \none
\end{ytableau}}
\end{equation}

The central question of the remainder of this present work is: 
{\bf Does there exist a single Rule $\JackjLR_{21,21}^{321}$ that solves this entire family}, i.e.
\begin{equation}\label{eq:321fameq}
\jackjLR_{\mu,21}^{\lambda} / F_{\mu,21}^{\lambda} = [\JackjLR_{21,21}^{321}]_R
\end{equation}

Note that the right hand side of \eqref{eq:321fameq} is manifestly independent of $\sigma_i$ and $n_0,m_0$.
Thus, the generic case of \eqref{eq:4paramwindowgeneral} reduces to a family of partitions $\mu',\lambda'$  shaped by 4-parameters ${\bf{v}} = (m_1,n_1,m_2,n_2)$, described by the following window factor

\begin{equation}\label{eq:4paramwindow}
\ytableausetup{boxsize=1.1em}
 \windfact_{R({\bf{v}})} = \frac{\begin{ytableau}
\none[m_2\,]   &*(boxU)  \\
\none   &*(boxS) a_1 &*(boxL) \\
\none[m_1\,]   &*(boxU)  & *(boxG)  & *(boxU)  \\
\none &   *(boxS) a_2 &*(boxL)  &*(boxS)  a_3& *(boxL)   \\
\none[\mu'] &  \none & \none[n_2]  & \none & \none[n_1]
\end{ytableau} \,\, \begin{ytableau}
\none\\
\none \\
\none & *(boxS) b_1  \\
\none & *(boxS) b_2  & *(boxS) b_3  \\
\none[\nu] & \none & \none 
\end{ytableau}
}{\begin{ytableau}
\none & *(boxS) c_1  \\
\none[m_2\,]   &*(boxU)\\
\none & *(boxS) c_2 &*(boxL)  & *(boxS) c_3  \\
\none[m_1\,]  &*(boxU) & *(boxG) &*(boxU)  \\
\none & *(boxS) c_4 &*(boxL)  &*(boxS)  c_5   & *(boxL)  & *(boxS) c_6 \\
\none[\lambda'] & \none  & \none[n_2] & \none & \none[n_1] & \none
\end{ytableau}}
\end{equation}
In the following sense
\begin{equation}
 \jackjLR_{\mu,21}^{\lambda} / F_{\mu,21}^{\lambda} = \jackjLR_{\mu', 21}^{\lambda'} / F_{R({\bf{v}})}.
\end{equation}
So our main question becomes: does there exist a single rule $\JackjLR_{21,21}^{321}$ such that
\begin{equation}\label{eq:equationsfamily}
\JackjLR_{\mu,21}^{\lambda}({\bf{v}}) = \windfact_{R({\bf{v}})} \circ (\JackjLR_{21,21}^{321})_{R({\bf{v}})} 
\end{equation}
is a Rule for $\jackjLR_{\mu,21}^{\lambda}({\bf{v}})$ for all ${\bf{v}}$. Such a rule would then give a full solution to the action of $J_{21}$, i.e. \eqref{eq:321fameq}.

We consider the equation \eqref{eq:equationsfamily} an infinite family of linear equations with rational coefficients for the $2^{12}$ variables $c_\bmD \in \BZ$, in $ \JackjLR_{21,21}^{321} = \sum_{\bmD} c_{\bmD}\cdot \bmD$ as follows. The family is infinite as we vary over the windows $R_{({\bf{v}})}$ of \eqref{eq:4paramwindow} with parameters  ${\bf{v}}=(m_1,n_1,m_2,n_2) \in \BZ_{\geq 0}^4$, and then we take the residues as functions of $\alpha$ at all of the possible poles of the equation. Specifically, our equations are

\begin{equation}\label{def:Zfamily}
\mathcal{Z} := \big\{ \Res_{\alpha\to z_b^\sB}^{(k)}\left( \jackjLR_{\mu\nu}^{\lambda}/F_{R} - \sum_{\bmD} c_{\bmD} [\bmD]_R \right) =0 \big \}_{A \in \mathcal{A}}
\end{equation}
where the range of parameters $({\bf{v}},z_b^\sB,k) =: A \in \mathcal{A}$ is given by 
\begin{itemize}
\item ${\bf{v}}=(m_1,n_1,m_2,n_2) \in \BZ_{\geq 0}^4 $ are the window parameters, i.e. the parameters of the relative positions of the components of $\lambda/\mu$.
\item According to the Window Positivity conjecture \ref{conj:windowpositivity} all the possible poles of $\jackjLR_{\mu\nu}^{\lambda}/F_{R} $ are given by the zeros of $j_{\lambda,R}$, that is, a zero of any one of the $h^\sA_{\lambda}$ factors, 
\begin{equation}
\{ z_b^\sB = - \tfrac{leg_{\lambda({\bf{v}})}(b^{R({\bf{v}})})+\delta_\sB^\sL}{arm_{\lambda({\bf{v}})}(b^{R({\bf{v}})})+\delta_\sB^\sU} : b\in \{321\}, \sB \in \{\sU,\sL\}  \}.
\end{equation}
Clearly, there are at most 12 possible zeros $z_b^\sB $ per value of ${\bf{v}}$.
\item $1 \leq k \leq ord_{z_b^\sB}(j_{\lambda,R})$ ranges up to the order of the zero in $j_{\lambda,R}$.

\end{itemize}

For example, for $g_{21,21}^{321} = \frac{6\alpha(2+11\alpha+2\alpha^2)}{(3+2\alpha)(2+3\alpha)(1+2\alpha)(2+\alpha)}$, the possible poles are 
\[ (z_b^\sB)_{ord} \in \{(-\tfrac{3}{2})_1, (-\tfrac{2}{3})_1, (-2)_2, (-\tfrac{1}{2})_2, (0)_3\}.\]

\subsubsection{The Solution} As the system of equations \eqref{def:Zfamily} is highly over-constrained, it is surprising that a solution exists, and even moreso that one exists over $\BZ$. The following Stanley Sum was found using Mathematica to solve the corresponding linear system of $\sim 15000$ equations (over a range of $\bf{v}$) in the $2^{12}$ variables $c_{\bmD}$.

\begin{conjecture}\label{conj:generalsolution}
The following Stanley sum $\JackjLR_{21,21}^{321} = \sum_{\bmD} c_{\bmD}\cdot \bmD$ solves the equations \eqref{def:Zfamily} for all ${\bf{v}} \in \BZ^4_{\geq 0}$ and has weight 0, 


\begin{eqnarray}\label{eq:generalsolution}
\ytableausetup{boxsize=0.4em}
\JackjLR_{21,21}^{321} &= & 7 \left( \frac{ \begin{ytableau}
*(boxU)  \\
*(boxU)    &  *(boxU) 
\end{ytableau}\,\,\begin{ytableau}
*(boxU)   \\
*(boxU)   &  *(boxU) 
\end{ytableau} }{ \begin{ytableau}
*(boxU) \\
*(boxU)   & *(boxL)  \\
*(boxU)  & *(boxU)   & *(boxL) 
\end{ytableau} } \right) \quad - 2 \left( \frac{ \begin{ytableau}
*(boxU)  \\
*(boxU)    &  *(boxU) 
\end{ytableau}\,\,\begin{ytableau}
*(boxU)   \\
*(boxU)   &  *(boxU) 
\end{ytableau} }{ \begin{ytableau}
*(boxU) \\
*(boxU)   & *(boxU)  \\
*(boxU)  & *(boxU)   & *(boxL) 
\end{ytableau} } \right) \\
&& +1 \left(  \frac{ \begin{ytableau}
*(boxU)  \\
*(boxL)    &  *(boxU) 
\end{ytableau}\,\,\begin{ytableau}
*(boxU)   \\
*(boxU)   &  *(boxU) 
\end{ytableau} }{ \begin{ytableau}
*(boxU) \\
*(boxU)   & *(boxU)  \\
*(boxU)  & *(boxU)   & *(boxL) 
\end{ytableau} } +  \frac{ \begin{ytableau}
*(boxU)  \\
*(boxU)    &  *(boxU) 
\end{ytableau}\,\,\begin{ytableau}
*(boxU)   \\
*(boxL)   &  *(boxU) 
\end{ytableau} }{ \begin{ytableau}
*(boxU) \\
*(boxU)   & *(boxU)  \\
*(boxU)  & *(boxU)   & *(boxL) 
\end{ytableau} }+  \frac{ \begin{ytableau}
*(boxU)  \\
*(boxU)    &  *(boxU) 
\end{ytableau}\,\,\begin{ytableau}
*(boxU)   \\
*(boxU)   &  *(boxU) 
\end{ytableau} }{ \begin{ytableau}
*(boxU) \\
*(boxU)   & *(boxU)  \\
*(boxL)  & *(boxU)   & *(boxL) 
\end{ytableau} }\right) \nonumber \\
&& - 2 \left( \frac{ \begin{ytableau}
*(boxL)  \\
*(boxU)    &  *(boxU) 
\end{ytableau}\,\,\begin{ytableau}
*(boxU)   \\
*(boxU)   &  *(boxU) 
\end{ytableau} }{ \begin{ytableau}
*(boxU) \\
*(boxU)   & *(boxL)  \\
*(boxU)  & *(boxU)   & *(boxL) 
\end{ytableau} } +\frac{ \begin{ytableau}
*(boxU)  \\
*(boxU)    &  *(boxL) 
\end{ytableau}\,\,\begin{ytableau}
*(boxU)   \\
*(boxU)   &  *(boxU) 
\end{ytableau} }{ \begin{ytableau}
*(boxU) \\
*(boxU)   & *(boxL)  \\
*(boxU)  & *(boxU)   & *(boxL) 
\end{ytableau} } +\frac{ \begin{ytableau}
*(boxU)  \\
*(boxU)    &  *(boxU) 
\end{ytableau}\,\,\begin{ytableau}
*(boxL)   \\
*(boxU)   &  *(boxU) 
\end{ytableau} }{ \begin{ytableau}
*(boxU) \\
*(boxU)   & *(boxL)  \\
*(boxU)  & *(boxU)   & *(boxL) 
\end{ytableau} } +\frac{ \begin{ytableau}
*(boxU)  \\
*(boxU)    &  *(boxU) 
\end{ytableau}\,\,\begin{ytableau}
*(boxU)   \\
*(boxU)   &  *(boxL) 
\end{ytableau} }{ \begin{ytableau}
*(boxL) \\
*(boxU)   & *(boxL)  \\
*(boxU)  & *(boxU)   & *(boxU) 
\end{ytableau} }+\frac{ \begin{ytableau}
*(boxU)  \\
*(boxU)    &  *(boxU) 
\end{ytableau}\,\,\begin{ytableau}
*(boxU)   \\
*(boxU)   &  *(boxU) 
\end{ytableau} }{ \begin{ytableau}
*(boxU) \\
*(boxL)   & *(boxL)  \\
*(boxU)  & *(boxU)   & *(boxL) 
\end{ytableau} } +\frac{ \begin{ytableau}
*(boxU)  \\
*(boxU)    &  *(boxU) 
\end{ytableau}\,\,\begin{ytableau}
*(boxU)   \\
*(boxU)   &  *(boxU) 
\end{ytableau} }{ \begin{ytableau}
*(boxU) \\
*(boxU)   & *(boxL)  \\
*(boxU)  & *(boxL)   & *(boxL) 
\end{ytableau} }  \right) \nonumber \\
&&  - 2 \left( \frac{ \begin{ytableau}
*(boxU)  \\
*(boxL)    &  *(boxU) 
\end{ytableau}\,\,\begin{ytableau}
*(boxU)   \\
*(boxU)   &  *(boxU) 
\end{ytableau} }{ \begin{ytableau}
*(boxU) \\
*(boxU)   & *(boxL)  \\
*(boxU)  & *(boxU)   & *(boxL) 
\end{ytableau} }+
 \frac{ \begin{ytableau}
*(boxU)  \\
*(boxU)    &  *(boxU) 
\end{ytableau}\,\,\begin{ytableau}
*(boxU)   \\
*(boxL)   &  *(boxU) 
\end{ytableau} }{ \begin{ytableau}
*(boxU) \\
*(boxU)   & *(boxL)  \\
*(boxU)  & *(boxU)   & *(boxL) 
\end{ytableau} } + 
 \frac{ \begin{ytableau}
*(boxU)  \\
*(boxU)    &  *(boxU) 
\end{ytableau}\,\,\begin{ytableau}
*(boxU)   \\
*(boxU)   &  *(boxU) 
\end{ytableau} }{ \begin{ytableau}
*(boxU) \\
*(boxU)   & *(boxL)  \\
*(boxL)  & *(boxU)   & *(boxL) 
\end{ytableau} } \right) \nonumber \\
&& +1 \left( \frac{ \begin{ytableau}
*(boxL)  \\
*(boxL)    &  *(boxU) 
\end{ytableau}\,\,\begin{ytableau}
*(boxU)   \\
*(boxU)   &  *(boxU) 
\end{ytableau} }{ \begin{ytableau}
*(boxU) \\
*(boxU)   & *(boxL)  \\
*(boxU)  & *(boxU)   & *(boxL) 
\end{ytableau} } +\frac{ \begin{ytableau}
*(boxU)  \\
*(boxL)    &  *(boxL) 
\end{ytableau}\,\,\begin{ytableau}
*(boxU)   \\
*(boxU)   &  *(boxU) 
\end{ytableau} }{ \begin{ytableau}
*(boxU) \\
*(boxU)   & *(boxL)  \\
*(boxU)  & *(boxU)   & *(boxL) 
\end{ytableau} } +\frac{ \begin{ytableau}
*(boxU)  \\
*(boxU)    &  *(boxU) 
\end{ytableau}\,\,\begin{ytableau}
*(boxL)   \\
*(boxL)   &  *(boxU) 
\end{ytableau} }{ \begin{ytableau}
*(boxU) \\
*(boxU)   & *(boxL)  \\
*(boxU)  & *(boxU)   & *(boxL) 
\end{ytableau} } +\frac{ \begin{ytableau}
*(boxU)  \\
*(boxU)    &  *(boxU) 
\end{ytableau}\,\,\begin{ytableau}
*(boxU)   \\
*(boxL)   &  *(boxL) 
\end{ytableau} }{ \begin{ytableau}
*(boxL) \\
*(boxU)   & *(boxL)  \\
*(boxU)  & *(boxU)   & *(boxU) 
\end{ytableau} }+\frac{ \begin{ytableau}
*(boxU)  \\
*(boxU)    &  *(boxU) 
\end{ytableau}\,\,\begin{ytableau}
*(boxU)   \\
*(boxU)   &  *(boxU) 
\end{ytableau} }{ \begin{ytableau}
*(boxU) \\
*(boxL)   & *(boxL)  \\
*(boxL)  & *(boxU)   & *(boxL) 
\end{ytableau} } +\frac{ \begin{ytableau}
*(boxU)  \\
*(boxU)    &  *(boxU) 
\end{ytableau}\,\,\begin{ytableau}
*(boxU)   \\
*(boxU)   &  *(boxU) 
\end{ytableau} }{ \begin{ytableau}
*(boxU) \\
*(boxU)   & *(boxL)  \\
*(boxL)  & *(boxL)   & *(boxL) 
\end{ytableau} }  \right) \nonumber \\
&& +1 \left( \frac{ \begin{ytableau}
*(boxL)  \\
*(boxU)    &  *(boxU) 
\end{ytableau}\,\,\begin{ytableau}
*(boxU)   \\
*(boxU)   &  *(boxL) 
\end{ytableau} }{ \begin{ytableau}
*(boxL) \\
*(boxU)   & *(boxL)  \\
*(boxU)  & *(boxU)   & *(boxU) 
\end{ytableau} } +
 \frac{ \begin{ytableau}
*(boxU)  \\
*(boxU)    &  *(boxL) 
\end{ytableau}\,\,\begin{ytableau}
*(boxL)   \\
*(boxU)   &  *(boxU) 
\end{ytableau} }{ \begin{ytableau}
*(boxU) \\
*(boxU)   & *(boxL)  \\
*(boxU)  & *(boxU)   & *(boxL) 
\end{ytableau} } +
 \frac{ \begin{ytableau}
*(boxL)  \\
*(boxU)    &  *(boxU) 
\end{ytableau}\,\,\begin{ytableau}
*(boxU)   \\
*(boxU)   &  *(boxU) 
\end{ytableau} }{ \begin{ytableau}
*(boxU) \\
*(boxU)   & *(boxL)  \\
*(boxU)  & *(boxL)   & *(boxL) 
\end{ytableau} }+
 \frac{ \begin{ytableau}
*(boxU)  \\
*(boxU)    &  *(boxL) 
\end{ytableau}\,\,\begin{ytableau}
*(boxU)   \\
*(boxU)   &  *(boxU) 
\end{ytableau} }{ \begin{ytableau}
*(boxU) \\
*(boxL)   & *(boxL)  \\
*(boxU)  & *(boxU)   & *(boxL) 
\end{ytableau} }+
 \frac{ \begin{ytableau}
*(boxU)  \\
*(boxU)    &  *(boxU) 
\end{ytableau}\,\,\begin{ytableau}
*(boxL)   \\
*(boxU)   &  *(boxU) 
\end{ytableau} }{ \begin{ytableau}
*(boxU) \\
*(boxU)   & *(boxL)  \\
*(boxU)  & *(boxL)   & *(boxL) 
\end{ytableau} }+
 \frac{ \begin{ytableau}
*(boxU)  \\
*(boxU)    &  *(boxU) 
\end{ytableau}\,\,\begin{ytableau}
*(boxU)   \\
*(boxU)   &  *(boxL) 
\end{ytableau} }{ \begin{ytableau}
*(boxL) \\
*(boxL)   & *(boxL)  \\
*(boxU)  & *(boxU)   & *(boxU) \nonumber
\end{ytableau} }
\right).
\end{eqnarray}
\ytableausetup{boxsize=0.9em}
\end{conjecture}
Importantly, we see that in this solution only $26$ of the $2^{12}$ coefficients $c_\bmD$ are non-vanishing. We can verify that this Rule has weight 0,
\begin{eqnarray*}
|\JackjLR_{21,21}^{321}| &=&7(-2)-2(-1)\\
&&+1(0+0-2) \\
&&-2(-1-1-1-1-3-3)\\
&&-2(-1-1-3)\\
&&+1(0+0+0+0-4-4)\\
&&+1(0+0-2-2-2-2) \\
&=& 0.
\end{eqnarray*}

The existence of the solution \eqref{eq:generalsolution} is strong evidence for the validity of the Jack windowing \eqref{conj:jackwindowing} conjecture.

\begin{example}\label{ex:hanloncalc}
We revisit the example \eqref{example:hanlonpositivity} that Stanley attributes to Hanlon.

\begin{equation}
 \jackjLR_{31,21}^{421} = \frac{\alpha( 9+97\alpha+294\alpha^2 + 321\alpha^3 + 131\alpha^4+12\alpha^5 )}{3(1+\alpha)^3(2+\alpha)(1+2\alpha)^2(1+3\alpha)} 
\end{equation}
For $R= \meetjoin(421/31)$, we have the window factor
\[ F_{R} = \frac{
\ytableausetup{boxsize=0.8em}
\begin{ytableau}
*(boxS)  \\
*(boxS) & *(boxS) & *(boxL) \sL    \\
\end{ytableau}
}{
\begin{ytableau}
*(boxS) \\
*(boxS) & *(boxS)\\
*(boxS) & *(boxS) & *(boxL) \sL & *(boxS)   \\
\end{ytableau}} = \frac{1}{1+\alpha}.
\]
In particular, this example falls into our main family \eqref{eq:4paramwindow} with $(m_1,m_2,n_1,n_2) = (0,0,1,0)$. By conjecture \eqref{eq:generalsolution} it is solved by our general solution $\JackjLR_{21,21}^{321}$, which we can evaluate
\begin{eqnarray} [\JackjLR_{21,21}^{321}]_R &=& 7\left(\tfrac{\alpha^3}{1+2\alpha}\right) - 2\left(\tfrac{\alpha^2}{1+2\alpha}\right)  \\
&& +1\left(\tfrac{2\alpha^2(1+\alpha)}{(1+2\alpha)(1+3\alpha)}+\tfrac{\alpha^2(2+\alpha)}{(1+2\alpha)^2}+\tfrac{2\alpha^2}{3(1+\alpha)} \right) \nonumber \\
&& -2\left( \tfrac{\alpha^2}{1+2\alpha} + \tfrac{\alpha^2(1+\alpha)}{2(1+2\alpha)}+\tfrac{\alpha^2}{1+2\alpha}+\tfrac{\alpha^2}{1+2\alpha}+\tfrac{\alpha^3}{2+\alpha}+\tfrac{\alpha^3(1+3\alpha)}{2(1+\alpha)(1+2\alpha)}\right) \nonumber\\
&& -2\left( \tfrac{2\alpha^3(1+\alpha)}{(1+2\alpha)(1+3\alpha)} + \tfrac{\alpha^3(2+\alpha)}{(1+2\alpha)^2} +\tfrac{2\alpha^3}{3(1+\alpha)}\right) \nonumber \\
&& +1\left( \tfrac{2\alpha^2(1+\alpha)}{(1+2\alpha)(1+3\alpha)} + \tfrac{\alpha^2(1+\alpha)^2}{(1+2\alpha)(1+3\alpha)} + \tfrac{\alpha^2(2+\alpha)}{(1+2\alpha)^2}+ \tfrac{\alpha^2(2+\alpha)}{(1+2\alpha)^2}+ \tfrac{2\alpha^3(1+2\alpha)}{3(1+\alpha)(2+\alpha)}    + \tfrac{\alpha^3(1+3\alpha)}{3(1+\alpha)^2}  \right) \nonumber \\
&& +1\left(  \tfrac{\alpha}{1+2\alpha}+\tfrac{\alpha(1+\alpha)}{2(1+2\alpha)}+\tfrac{\alpha^2(1+3\alpha)}{2(1+\alpha)(1+2\alpha)}+\tfrac{\alpha^2(1+\alpha)}{2(2+\alpha)}+\tfrac{\alpha^2(1+3\alpha)}{2(1+\alpha)(1+2\alpha)}+\tfrac{\alpha^2}{2+\alpha}  \right) \nonumber \\
&=& \frac{\alpha( 9+97\alpha+294\alpha^2 + 321\alpha^3 + 131\alpha^4+12\alpha^5 )}{3(1+\alpha)^2(2+\alpha)(1+2\alpha)^2(1+3\alpha)}.
\end{eqnarray}
This agrees with $\jackjLR_{31,21}^{421}/F_R$ as expected.
\end{example}

\subsection{Window Parameters}
We saw that in the family of \eqref{eq:4paramwindow} we had the four parameters $(m_1,m_2,n_1,n_2)$. Here we briefly discuss about how such families are generated. 
Let $\sigma=(\mu,\nu,\lambda)$ be a triple of partitions. Here we consider \emph{composite} Windowings, where we can window over $\lambda/\mu$ and $\lambda/\nu$, in either order.
The family of partitions which reduce via composite windows to $\sigma$ has at most a number of parameters computed by counting the number of connected components $\#(\lambda/\mu)+\#(\lambda/\nu)+2$. We label the boxes which separate such connected components with $i$, we have at most two windowing parameters, $m_i, n_i$, for the vertical and horizontal spacing introduced at that point.

\begin{example} Consider $\lambda = 33221$, $\mu=2211$, $\nu=2111$, there are two points separating components of $\lambda/\mu$, labelled $(1)$ and $(2)$, and one point separating $\lambda/\nu$, labelled $(3)$ below:
\ytableausetup{boxsize=1.0em}
\[
\begin{ytableau}
 *(boxS) & \none[1] \\
 *(boxB) & *(boxS) \\
  *(boxB)  & *(boxS) & \none[2]\\
  *(boxB)  &*(boxB) & *(boxS)  \\
 *(boxB) &*(boxB)   & *(boxS) 
\end{ytableau} \qquad 
\begin{ytableau}
 *(boxS) & \none[3]\\
 *(boxB) & *(boxS) \\
  *(boxB)  & *(boxS) \\
  *(boxB)  &*(boxS) & *(boxS)  \\
 *(boxB) &*(boxB)   & *(boxS) 
\end{ytableau}
\]
\end{example}

For these two composite windowings, there are two corresponding windowing factors for the corresponding families of partitions.
\begin{equation}\label{eq:fff}
\windfact_{R_{12|3}} =  \frac{\begin{ytableau}
\none[m_1\,\,]  &*(boxU) & \none[1] \\
\none &*(boxS)  &*(boxL) \\
\none &*(boxS)  & *(boxL) \\
\none[m_2\,\,]   &*(boxU)  & *(boxG)  & *(boxU) & \none[2] \\
\none  &*(boxS)  &*(boxL)  &*(boxS)  & *(boxL)   \\
\none  &*(boxS) & *(boxL)   &*(boxS)  & *(boxL)  \\
\none  &\none & \none[n_1]  & \none & \none[n_2]
\end{ytableau} \,\, \begin{ytableau}
\none \\
\none[m_3\,\,] & *(boxU)& \none[3]  \\
\none & *(boxS)  \\
\none & *(boxS)  \\
\none & *(boxS)   \\
\none &  *(boxS) & *(boxS)   \\
\none & \none  & \none 
\end{ytableau}
}{\begin{ytableau}
\none & *(boxS) \\
\none[m_3\,\,] & *(boxU)& \none[3]  \\
\none[m_1\,\,] & *(boxU) & \none[1]  \\
\none  & *(boxS) & *(boxL) & *(boxS) \\
\none & *(boxS) &*(boxL) & *(boxS) \\
\none[m_2\,\,] & *(boxU) & *(boxG) &*(boxU)& \none[2]   \\
\none & *(boxS) &*(boxL)  &*(boxS) & *(boxL)   & *(boxS)  \\
\none & *(boxS) &*(boxL)  &*(boxS) & *(boxL)  & *(boxS) \\
\none & \none  & \none[n_1] & \none  & \none[n_2] & \none
\end{ytableau}}, \qquad \windfact_{R_{3|12}} =
 \frac{\begin{ytableau}
 \none \\
\none[m_1\,\,]  &*(boxU) & \none[1] \\
\none &*(boxS)  \\
\none &*(boxS)  &  \none & \none[2] \\
\none  &*(boxS)    &*(boxS)  & *(boxL)   \\
\none  &*(boxS)  &*(boxS)  & *(boxL)  \\
\none  &\none   & \none & \none[n_2]
\end{ytableau} \,\, \begin{ytableau}
\none \\
\none[m_3\,\,] & *(boxU)& \none[3]  \\
\none & *(boxS) &*(boxL)  \\
\none & *(boxS) &*(boxL)  \\
\none & *(boxS) &*(boxL)    \\
\none &  *(boxS)&*(boxL)  & *(boxS)   \\
\none & \none  & \none[n_3] & \none 
\end{ytableau}
}{\begin{ytableau}
\none & *(boxS) \\
\none[m_1\,\,] & *(boxU) & \none[1]  \\
\none[m_3\,\,] & *(boxU)& \none[3]  \\
\none  & *(boxS) & *(boxL) & *(boxS) \\
\none & *(boxS) &*(boxL) & *(boxS) & \none[2]   \\
\none & *(boxS) &*(boxL)  &*(boxS) & *(boxL)   & *(boxS)  \\
\none & *(boxS) &*(boxL)  &*(boxS) & *(boxL)  & *(boxS) \\
\none & \none  & \none[n_3] & \none  & \none[n_2] & \none
\end{ytableau}}.
\end{equation}

\subsection{Beyond $c=2$}

Unfortunately, the particular example of $c_{21,21}^{321}$ may be the only one that we can solve computationally, as there are $2^{12}$ unknowns. For a small $c=3$ example like $c_{321,221}^{43211}$ there are $2^{22} = \,\sim4\times 10^6$ variables $c_\bmD$, for which one would be considering the family:

\begin{equation}
\frac{\begin{ytableau}
\none \\
\none & *(boxS)  \\
\none[m_1\,\,] & *(boxL) &*(boxL) & \none[1] \\
\none & *(boxS) &*(boxS) & *(boxU)    \\
\none[m_2\,\,] & *(boxL) &*(boxL) & *(boxG) & *(boxL) & \none[2]  \\
\none & *(boxS) &*(boxS) & *(boxU) & *(boxS) & *(boxU)      \\
\none & \none  & \none & \none[n_1]  & \none & \none[n_2] 
\end{ytableau}\,
\begin{ytableau}
\none \\
\none & *(boxS) &*(boxS)  \\
\none & *(boxS) &*(boxS)  & *(boxS)   \\
\none & \none  & \none & \none 
\end{ytableau}
}{
\begin{ytableau}
\none \\
\none & *(boxS)  \\
\none & *(boxS)  \\
\none & *(boxS) &*(boxS) \\
\none[m_1\,\,] & *(boxL) &*(boxL) & \none[1] \\
\none & *(boxS) &*(boxS) & *(boxU) & *(boxS)  \\
\none[m_2\,\,] & *(boxL) &*(boxL) & *(boxG) & *(boxL) & \none[2]  \\
\none & *(boxS) &*(boxS) & *(boxU) & *(boxS) & *(boxU) &  *(boxS)   \\
\none & \none  & \none & \none[n_1]  & \none & \none[n_2] 
\end{ytableau}}
\end{equation}

\subsection{The Kernel}

The system of equations \eqref{def:Zfamily} has a large kernel. For example, consider the Stanley sum
\begin{equation}
\ytableausetup{boxsize=0.7em}
 \bk := \frac{ \begin{ytableau}
   \\
     &   
\end{ytableau}\,\,\begin{ytableau}
   \\
    &  
\end{ytableau} }{ \begin{ytableau}
  \\
    & *(boxL) \sL  \\
   &     & *(boxU)\sU
\end{ytableau} } - \frac{ \begin{ytableau}
   \\
     &  
\end{ytableau}\,\,\begin{ytableau}
    \\
   &  
\end{ytableau} }{ \begin{ytableau}
  \\
    & *(boxU) \sU \\
   &     & *(boxL)  \sL
\end{ytableau} },
\ytableausetup{boxsize=0.9em}
\end{equation}
where the unfilled boxes are prescribed any fixed set of hook choices (the same for each of the two terms).
Then clearly
\[ [\bk]_R = 0 \]
for any $R$ in the family \eqref{eq:4paramwindow}. More complicated examples can be shown to exist.

\begin{example} The following Stanley sum is in the kernel of the family \eqref{eq:4paramwindow}
\begin{eqnarray}
\label{eq:kernelfirst}
\ytableausetup{boxsize=0.6em}
\bmk' &=&
\,\,\,\,\,1 \cdot \frac{ \begin{ytableau}
*(boxU)  \\
*(boxU)    &  *(boxU) 
\end{ytableau}\,\,\begin{ytableau}
  \\
  &  
\end{ytableau} }{ \begin{ytableau}
 \\
    &  *(boxL) \\
   &    &  
\end{ytableau} } 
-1 \cdot \frac{ \begin{ytableau}
*(boxU)  \\
*(boxU)    &  *(boxL) 
\end{ytableau}\,\,\begin{ytableau}
  \\
  &  
\end{ytableau} }{ \begin{ytableau}
 \\
    &  *(boxU) \\
   &    &  
\end{ytableau} }
+1 \cdot
\frac{ \begin{ytableau}
*(boxU)  \\
*(boxL)    &  *(boxU) 
\end{ytableau}\,\,\begin{ytableau}
  \\
  &  
\end{ytableau} }{ \begin{ytableau}
 \\
    &  *(boxU) \\
   &    &  
\end{ytableau} }
- 2\cdot
\frac{ \begin{ytableau}
*(boxU)  \\
*(boxL)    &  *(boxU) 
\end{ytableau}\,\,\begin{ytableau}
  \\
  &  
\end{ytableau} }{ \begin{ytableau}
 \\
    &  *(boxL) \\
   &    &  
\end{ytableau} }  \\
&&
+\,1 \cdot \frac{ \begin{ytableau}
*(boxU)  \\
*(boxL)    &  *(boxL) 
\end{ytableau}\,\,\begin{ytableau}
  \\
  &  
\end{ytableau} }{ \begin{ytableau}
 \\
    &  *(boxL) \\
   &    &  
\end{ytableau} } 
-1 \cdot \frac{ \begin{ytableau}
*(boxL)  \\
*(boxU)    &  *(boxU) 
\end{ytableau}\,\,\begin{ytableau}
  \\
  &  
\end{ytableau} }{ \begin{ytableau}
 \\
    &  *(boxU) \\
   &    &  
\end{ytableau} }
+2 \cdot
\frac{ \begin{ytableau}
*(boxL)  \\
*(boxU)    &  *(boxL) 
\end{ytableau}\,\,\begin{ytableau}
  \\
  &  
\end{ytableau} }{ \begin{ytableau}
 \\
    &  *(boxU) \\
   &    &  
\end{ytableau} }
-1\cdot
\frac{ \begin{ytableau}
*(boxL)  \\
*(boxU)    &  *(boxL) 
\end{ytableau}\,\,\begin{ytableau}
  \\
  &  
\end{ytableau} }{ \begin{ytableau}
 \\
    &  *(boxL) \\
   &    &  
\end{ytableau} }  \nonumber \\
&&
+\,1 \cdot \frac{ \begin{ytableau}
*(boxL)  \\
*(boxL)    &  *(boxU) 
\end{ytableau}\,\,\begin{ytableau}
  \\
  &  
\end{ytableau} }{ \begin{ytableau}
 \\
    &  *(boxL) \\
   &    &  
\end{ytableau} } 
-1 \cdot \frac{ \begin{ytableau}
*(boxL)  \\
*(boxL)    &  *(boxL) 
\end{ytableau}\,\,\begin{ytableau}
  \\
  &  
\end{ytableau} }{ \begin{ytableau}
 \\
    &  *(boxU) \\
   &    &  
\end{ytableau} }. \nonumber
\end{eqnarray}
\ytableausetup{boxsize=0.8em}
where the empty boxes are filled with any fixed choice of hooks. 
\end{example}
We will describe the kernel in more detail in a follow up work \cite{Mickler:2025} .

\section{Factorization}\label{section:factorization}

\ytableausetup{boxsize=0.7em}

We have shown that the windowing property of Schur LR coefficients \eqref{lemma:schurlrwindow} can be extended to Jack LR coefficients \eqref{conj:jackwindowing}. In this section we propose that a factorization property of Schur LR coefficients due to King, Tollu and Toumazet (KTT) can also be extended to the Jack case.

\subsection{The KTT Setup}

Consider an ordered subset $I \subset \{1,2,\ldots\}$ (resp $J,K$), whose entries refer to rows of $\mu$ (resp. $\nu,\lambda$). let $p(I)$ be the partition with rows $\{I_n-n : 1 \leq n \leq |I|\}$ in reverse order. E.g. $p(\{2,4,7\}) = \{4,2,1\}$. If $|I|=|J|=|K|$ and $c_{p(I),p(J)}^{p(K)} >0$, then we say that $\{I,J,K\}$ is a \emph{Horn triple}. We say the triple is \emph{essential} if $c_{p(I),p(J)}^{p(K)} = 1$. We define the partial sum $ps(\mu)_I := \{ \mu_i : i\in I\}$, and the restricted partition $\mu_I := \{ \mu_i : i \in I\}$. I.e. $ps(\mu)_I = |\mu_I|$.

\begin{theorem}[Schur Factorization: King, Tollu, Toumazet {\cite[1.4]{King:2009}} ]\label{thm:factorization}

Associated to an essential Horn triple $(I,J,K)$, and for a triple of partitions $(\mu,\nu,\lambda)$ with $|\mu|+|\nu|=|\lambda|$ such that
\begin{equation}
ps(\mu)_I + ps(\nu)_J = ps(\lambda)_K,
\end{equation}
there is a associated factorization of Schur LR coefficients:
\begin{equation}\label{eq:factorizationlr}
c^{\lambda}_{\mu,\nu} = c^{\lambda_K}_{\mu_I,\nu_J} \times c^{\lambda_{\bar K}}_{\mu_{\bar I},\nu_{\bar J}}.
\end{equation}
where $\bar I$, $\bar J$, $\bar K$ are the complements of $I$, $J$, $K$.
\end{theorem}

\begin{example}[{\cite{King:2009}}]\label{ex:kingfact}
Consider the essential Horn triple 
\begin{equation}
(I,J,K) =(\{1,2,4\},\{2,3,4\},\{2,3,5\}).
\end{equation}
The partitions $\mu = 2221$, $\nu=4211$, and $\lambda=44322$ satisfy the conditions of KTT factorization \eqref{thm:factorization} w.r.t $(I,J,K)$, and thus there is the factorization
\begin{equation}
c_{\underline{2}\underline{2}2\underline{1},4\underline{2}\underline{1}\underline{1}}^{4\underline{4}\underline{3}2\underline{2}} = c_{221,211}^{432} \times c_{2,4}^{42},
\end{equation}
where we use underlining to indicate the factorization (i.e. the terms associated with the Horn triple). Note that each of these Schur LR coefficients are equal to 1. This example is represented pictorially as (dots are use to distinguish the factors)
\begin{equation}
\frac{ \begin{ytableau}
 \mk \\
 &  \\
 \mk &  \mk  \\
 \mk &  \mk 
\end{ytableau} \,\, \begin{ytableau}
 \mk \\
 \mk \\
 \mk &  \mk  \\
 &  &  &   
\end{ytableau} } { \begin{ytableau}
 \mk & \mk \\
 &   \\
 \mk & \mk &  \mk  \\
 \mk & \mk &  \mk & \mk  \\
 &  &  &  
\end{ytableau} } 
= 
 \frac{ \begin{ytableau}
 \mk \\
 \mk &   \mk \\
 \mk &  \mk 
\end{ytableau} \,\, \begin{ytableau}
 \mk \\
 \mk \\
 \mk &  \mk 
\end{ytableau} } { \begin{ytableau}
 \mk & \mk \\
 \mk & \mk &  \mk  \\
 \mk & \mk &  \mk &  \mk 
\end{ytableau} } 
\times
\frac{ \begin{ytableau}
\, &  
\end{ytableau} \,\, \begin{ytableau}
\, &  &  &   
\end{ytableau} } { \begin{ytableau}
\, &   \\
 &  &  &  
\end{ytableau} }.
\end{equation}
\end{example}

We now conjecture that an analogous factorization result holds for Jack LR coefficients. We begin with an illustrative example.

\begin{example}
Returning to example \ref{ex:kingfact}, we compute the corresponding Jack LR coefficient
\begin{equation}
g_{2221,4211}^{44322} = \frac{3 \alpha^2 (3+\alpha)(4+\alpha)(1+3\alpha)(3+4\alpha)}{2(1+\alpha)^5(2+\alpha)^2(3+2\alpha)(4+3\alpha)^2}.
\end{equation}
We find this coefficient is given by the evaluation of a single Stanley Diagram (as predicted by the strong Stanley conjecture)
\begin{equation}
\JackjLR_{\underline{2}\underline{2}2\underline{1},4\underline{2}\underline{1}\underline{1}}^{4\underline{4}\underline{3}2\underline{2}} = \frac{ \begin{ytableau}
*(boxL) \\
*(boxU) & *(boxU) \\
*(boxL) & *(boxL) \\
*(boxL) & *(boxL)   
\end{ytableau} \,\, \begin{ytableau}
*(boxL) \\
*(boxL) \\
*(boxL) & *(boxL) \\
*(boxU) & *(boxU) & *(boxU) & *(boxU)  
\end{ytableau} } { \begin{ytableau}
*(boxL) & *(boxL) \\
*(boxU) & *(boxU) \\
*(boxL) & *(boxL) & *(boxL) \\
*(boxL) & *(boxL) & *(boxL) & *(boxL) \\
*(boxU) & *(boxU) & *(boxU) & *(boxU)  
\end{ytableau} }. 
\end{equation}
The form of this diagram suggests a {\bf\emph{factorization}} of the diagram
\begin{equation}
\frac{ \begin{ytableau}
*(boxL) \mk \\
*(boxU) & *(boxU) \\
*(boxL) \mk &*(boxL)  \mk  \\
*(boxL) \mk &*(boxL)  \mk 
\end{ytableau} \,\, \begin{ytableau}
*(boxL) \mk \\
*(boxL) \mk \\
*(boxL) \mk & *(boxL) \mk  \\
*(boxU) & *(boxU) & *(boxU) & *(boxU)  
\end{ytableau} } { \begin{ytableau}
*(boxL) \mk &*(boxL) \mk \\
*(boxU) & *(boxU)  \\
*(boxL) \mk &*(boxL) \mk &*(boxL)  \mk  \\
*(boxL) \mk &*(boxL) \mk &*(boxL)  \mk &*(boxL) \mk  \\
*(boxU) & *(boxU) & *(boxU) & *(boxU)  
\end{ytableau} } 
= 
 \frac{ \begin{ytableau}
*(boxL) \mk \\
*(boxL) \mk &*(boxL)   \mk \\
*(boxL) \mk &*(boxL)  \mk 
\end{ytableau} \,\, \begin{ytableau}
*(boxL) \mk \\
*(boxL) \mk \\
*(boxL) \mk & *(boxL)  \mk 
\end{ytableau} } { \begin{ytableau}
*(boxL) \mk &*(boxL) \mk \\
*(boxL) \mk &*(boxL) \mk &*(boxL)  \mk  \\
*(boxL) \mk &*(boxL) \mk &*(boxL)  \mk &*(boxL)  \mk 
\end{ytableau} } 
\circ
\frac{ \begin{ytableau}
*(boxU) & *(boxU) 
\end{ytableau} \,\, \begin{ytableau}
*(boxU) & *(boxU) & *(boxU) & *(boxU)  
\end{ytableau} } { \begin{ytableau}
*(boxU) & *(boxU)  \\
*(boxU) & *(boxU) & *(boxU) & *(boxU)  
\end{ytableau} },
\end{equation}
corresponding to the KTT factorization. We find that the two factors are also single-term Rules  
\begin{equation}
\JackjLR_{\underline{2}\underline{2}2\underline{1},4\underline{2}\underline{1}\underline{1}}^{4\underline{4}\underline{3}2\underline{2}} = \JackjLR_{221,211}^{432} \circ \JackjLR_{2,4}^{42}.
\end{equation}
\end{example}

We now conjecture that that such a factorization always exists given certain restrictions

\begin{conjecture}[Jack Factorization]\label{conj:factorization}
For a factorization of Schur LR coefficients associated to an essential Horn triple $(I,J,K)$ as in \eqref{eq:factorizationlr}, i.e.
\begin{equation}
c^{\lambda}_{\mu,\nu} = c^{\lambda_K}_{\mu_I,\nu_J} \times c^{\lambda_{\bar K}}_{\mu_{\bar I},\nu_{\bar J}},
\end{equation}
with the additional properties that
\begin{itemize}
\item $c^{\lambda_{ K}}_{\mu_{ I},\nu_{ J}} = 1$ (i.e. for the factor associated to the Horn triple)
 \item None of the entries of $\lambda_I$, $\mu_I$ or $\nu_K$ are $0$.
\end{itemize}
then there exists a triple of rules $\JackjLR^{\lambda}_{\mu,\nu}$, $\JackjLR^{\lambda_K}_{\mu_I,\nu_J}$, $\JackjLR^{\lambda_{\bar K}}_{\mu_{\bar I},\nu_{\bar J}}$ that satisfy a composition rule
\begin{equation}
\JackjLR^{\lambda}_{\mu,\nu}  = \JackjLR^{\lambda_K}_{\mu_I,\nu_J} \circ \JackjLR^{\lambda_{\bar K}}_{\mu_{\bar I},\nu_{\bar J}}.
\end{equation}
where the composition is given by the union of the Diagrams extended by linearity.
\end{conjecture}

\begin{example}
Consider the factorization
\begin{equation}
c_{\underline{2}2,2\underline{1}1}^{3\underline{3}2} = c_{2,1}^{3} \cdot c_{2,21}^{32},
\end{equation}
w.r.t the essential Horn triple $(I,J,K)=(\{1\},\{2\},\{2\})$, (underlining indicates the factorization). The corresponding Jack LR coefficients are
\[ g_{22,211}^{332} = \tfrac{\alpha(3+\alpha)}{(1+\alpha)^3(3+2\alpha)},\qquad  g_{2,1}^{3} = \tfrac{1}{1+2\alpha},\qquad  g_{2,21}^{32} = \tfrac{\alpha(2+\alpha)}{(1+\alpha)^2(1+2\alpha)}.  \]
From these, we can deduce we must have 
\[ g_{22,211}^{332} = 
 \frac{  \begin{ytableau}
 *(boxG) \\
 *(boxG)   \\
*(boxL) &  *(boxG)
\end{ytableau}\,\,\begin{ytableau}
*(boxG) &  *(boxG) \\
*(boxG) \sA &  *(boxG)  
\end{ytableau} } { \begin{ytableau}
*(boxG) & *(boxG) \\
*(boxL)   & *(boxG)\sA   & *(boxG)   \\
*(boxL) & *(boxU)  &  *(boxG) 
\end{ytableau}  } \times \frac{4\alpha}{(1+\alpha)}
\]
\[ 
g_{2,1}^{3}= \frac{ \begin{ytableau}
*(boxG)  &  *(boxG) 
\end{ytableau} \,\, \begin{ytableau}
 *(boxG)  \\
\end{ytableau} } { \begin{ytableau}
*(boxL)  & *(boxG)  & *(boxG) \\
\end{ytableau} } \times 1
\]
\[
g_{2,21}^{32} = \frac{ \begin{ytableau}
*(boxU) &  *(boxG) \\
\end{ytableau} \,\, \begin{ytableau}
 *(boxG) \\
*(boxL) &  *(boxG)
\end{ytableau} } { \begin{ytableau}
*(boxL) & *(boxG) \\
*(boxL) & *(boxU)  &  *(boxG)
\end{ytableau} } \times 1
\]

We find the following rules compatible with the above coefficients, and with both the factorization and windowing 

\[ \JackjLR_{2\underline{1}1,\underline{2}2}^{3\underline{3}2} = 
 \frac{\begin{ytableau}
 *(boxL) \\
 *(boxL) \mk  \\
*(boxL) &  *(boxU)
\end{ytableau}\,\, \begin{ytableau}
*(boxU) &  *(boxU) \\
*(boxU) \mk &  *(boxU) \mk 
\end{ytableau}   } { \begin{ytableau}
*(boxL) & *(boxU) \\
*(boxL) \mk  & *(boxU) \mk   & *(boxU)  \mk \\
*(boxL) & *(boxU)  &  *(boxU)
\end{ytableau} } 
=
 \frac{ \begin{ytableau}
 *(boxL) \mk \\
\end{ytableau} \, \,\begin{ytableau}
*(boxU) \mk &  *(boxU) \mk
\end{ytableau} } { \begin{ytableau}
*(boxL) \mk & *(boxU) \mk & *(boxU) \mk\\
\end{ytableau} } 
\circ
 \frac{ \begin{ytableau}
 *(boxL) \\
*(boxL) &  *(boxU)
\end{ytableau} \,\, \begin{ytableau}
*(boxU) &  *(boxU) \\
\end{ytableau} } { \begin{ytableau}
*(boxL) & *(boxU) \\
*(boxL) & *(boxU)  &  *(boxU)
\end{ytableau} }  = \JackjLR_{\underline{1},\underline{2}}^{\underline{3}} \circ \JackjLR_{21,2}^{32}.
\]
\end{example}
\begin{example}
Similarly, if we exchange the two factors $\mu,\nu$ in the previous example, we find the factorization

\[ \JackjLR_{\underline{2}2,2\underline{1}1}^{3\underline{3}2} = 
 \frac{ \begin{ytableau}
*(boxU) &  *(boxG) \\
*(boxL) \mk &  *(boxL) \mk 
\end{ytableau} \,\, \begin{ytableau}
 *(boxG) \\
 *(boxL) \mk  \\
*(boxL) &  *(boxG)
\end{ytableau} } { \begin{ytableau}
*(boxL) & *(boxG) \\
*(boxL) \mk  & *(boxL) \mk   & *(boxL)  \mk \\
*(boxL) & *(boxU)  &  *(boxL)
\end{ytableau} } 
=
 \frac{ \begin{ytableau}
*(boxL) \mk &  *(boxL) \mk
\end{ytableau} \,\, \begin{ytableau}
 *(boxL) \mk \\
\end{ytableau} } { \begin{ytableau}
*(boxL) \mk & *(boxL) \mk & *(boxL) \mk\\
\end{ytableau} } 
\circ
 \frac{ \begin{ytableau}
*(boxU) &  *(boxG) \\
\end{ytableau} \,\, \begin{ytableau}
 *(boxG) \\
*(boxL) &  *(boxG)
\end{ytableau} } { \begin{ytableau}
*(boxL) & *(boxG) \\
*(boxL) & *(boxU)  &  *(boxL)
\end{ytableau} } = \JackjLR_{\underline{2},\underline{1}}^{\underline{3}} \circ \JackjLR_{2,21}^{32}  
\]
where each of the set of four white boxes contributes a factor of $1$.
\end{example}

\begin{example}
The transpose of the above example factorizes over the same Horn triple
\[   \JackjLR_{\underline{2}2,3\underline{1}}^{3\underline{3}2} = \JackjLR_{2,1}^{3} \circ \JackjLR_{2,3}^{32}  \] 
\begin{equation}
 \frac{ \begin{ytableau}
*(boxU) &  *(boxU) \\
*(boxL) \mk  &  *(boxL) \mk 
\end{ytableau} \,\, \begin{ytableau}
 *(boxL)  \mk \\
*(boxU) &  *(boxU) & *(boxU) \\
\end{ytableau} } { \begin{ytableau}
*(boxU) & *(boxU) \\
*(boxL) \mk  & *(boxL)   \mk & *(boxL)  \mk \\
*(boxU) & *(boxU)  &  *(boxU)
\end{ytableau} } 
=
 \frac{ \begin{ytableau}
*(boxL) \mk  &  *(boxL) \mk 
\end{ytableau} \,\, \begin{ytableau}
 *(boxL) \mk  \\
\end{ytableau} } { \begin{ytableau}
*(boxL) \mk  & *(boxL)  \mk  & *(boxL) \mk  \\
\end{ytableau} } 
\circ
 \frac{ \begin{ytableau}
*(boxU) &  *(boxU) \\
\end{ytableau} \,\, \begin{ytableau}
*(boxU) &  *(boxU) & *(boxU) \\
\end{ytableau} } { \begin{ytableau}
*(boxU) & *(boxU) \\
*(boxU) & *(boxU)  &  *(boxU)
\end{ytableau} } 
\end{equation}
\end{example}

\begin{example}
For the factorization
\begin{equation}
c^{\underline{4}321}_{\underline{2}21,\underline{2}21} = c^{4}_{2,2} \cdot  c^{321}_{21,21} .
\end{equation}
associated to the essential Horn triple $(I,J,K)=(\{1\},\{1\},\{1\})$, we find
\begin{equation}  g_{221,221}^{4321} =  \frac{(3+\alpha)^22^2}{(4+3\alpha)(3+2\alpha)(2+\alpha)(1)} \cdot g_{21,21}^{321},
\end{equation}
which corresponds to the factorization
\begin{equation}   \frac{ \begin{ytableau}
  \\
   & \\
   *(boxL) &  *(boxL)   \\
\end{ytableau}\,\,\begin{ytableau}
  \\
   & \\
   *(boxL) &  *(boxL)   \\
\end{ytableau} }{ \begin{ytableau}
\\
  &  \\
 &    & \\
   *(boxL) &   *(boxL) &  *(boxL)  & *(boxL)   \\
\end{ytableau} }  =   \frac{ \begin{ytableau}
 *(boxL) &  *(boxL) 
\end{ytableau}\,\,\begin{ytableau}
 *(boxL)  &  *(boxL) 
\end{ytableau} }{ \begin{ytableau}
*(boxL) &  *(boxL)  &  *(boxL)  &  *(boxL) 
\end{ytableau} } 
\circ
\frac{ \begin{ytableau}
 \\
   &  
\end{ytableau}\,\,\begin{ytableau}
  \\
  &  
\end{ytableau} }{ \begin{ytableau}
\\
  &  \\
 &   & 
\end{ytableau} }  
\end{equation}
\end{example}

\ytableausetup{boxsize=1.0em}

\begin{lemma}
The Pieri rule is compatible with Factorization. That is, for $\nu = 1^r$ and $\lambda/\mu$ a vertical $r$-strip, we take the essential Horn triple $(I,J,K)=(\{i\},\{1\},\{i\})$ for any $i \leq r$. If $\lambda/\mu$ has a box in the $i$-th row, then we have $ps(\mu)_I +ps(\nu)_J = ps(\lambda)_K$ and hence we have a factorization
\begin{equation} \JackjLR_{\mu,1^r}^{\lambda} = \JackjLR_{\mu_i,1}^{\lambda_i} \circ \JackjLR_{\mu_{\bar I},1^{r-1}}^{\lambda_{\bar K}}.
\end{equation}
Here the removed row is 
\[  \JackjLR_{\mu_i,1}^{\lambda_i} = \frac{ \begin{ytableau}
*(boxL) & *(boxL) \cdots & *(boxL)
\end{ytableau}\,\,\begin{ytableau}
*(boxL)
\end{ytableau} }{ \begin{ytableau}
*(boxL) &   *(boxL)\cdots &  *(boxL)  & *(boxL)   \\
\end{ytableau} }  \] 
Furthermore, one can factor multiple rows at a time with 
\[(I,J,K)=(\{i_1,i_2,\ldots,i_n\},\{1,2,\ldots,n\},\{i_1,i_2,\ldots,i_n\}).\]

\end{lemma}

\subsection{Back to the $c=2$ problem}

We now consider certain factorizations which we expect to constrain the rule $\JackjLR_{21,21}^{321}$.

\begin{example}
Consider the factorization
\begin{equation}
c^{32\underline{2}1}_{2\underline{1}1,2\underline{1}1} = c^{2}_{1,1} c^{321}_{21,21}.
\end{equation}
associated to the essential Horn triple $(I,J,K)=(\{2\},\{2\},\{3\})$.

Thus, we expect a composite rule of the form:
\ytableausetup{boxsize=0.8em}
\begin{equation} 
 \frac{ \begin{ytableau}
 \\
*(boxL)   \\
   &  
\end{ytableau}\,\,\begin{ytableau}
  \\
 *(boxL)   \\
   &
\end{ytableau} }{ \begin{ytableau}
\\
*(boxL) &  *(boxL)   \\
  &  \\
 &    &
\end{ytableau} }=
 \frac{ \begin{ytableau}
*(boxL) 
\end{ytableau}\,\,\begin{ytableau}
 *(boxL) 
\end{ytableau} }{ \begin{ytableau}
*(boxL) &  *(boxL) 
\end{ytableau} } 
\circ
\frac{ \begin{ytableau}
 \\
   &  
\end{ytableau}\,\,\begin{ytableau}
  \\
  &  
\end{ytableau} }{ \begin{ytableau}
\\
  &  \\
 &   & 
\end{ytableau} } .
\end{equation}
\end{example}

\begin{example}
Consider the factorization
\begin{equation}
c^{3\underline{3}2\underline{2}1}_{2\underline{1}\underline{1}1,\underline{2}2\underline{1}1} =  c^{32}_{11,21} c^{321}_{21,21}.
\end{equation}
associated to the essential Horn triple $(I,J,K)=(\{3,2\},\{3,1\},\{4,2\})$. This example can be extended to a factorization of the form:

\begin{equation}
 \frac{ \begin{ytableau}
\none &    \\
\none[r_1] &   *(boxL)    \\
\none &     & \\
\none[r_2] &     *(boxL) &  *(boxL)   \\
\end{ytableau}  \,\, \begin{ytableau}
\none &   \\
\none[r_1] &  *(boxL)   \\
\none[r_2] &     *(boxL)     \\
\none &   &   \\
\end{ytableau}}{ \begin{ytableau}
\none &  \\
\none[r_1] &  *(boxL) &   *(boxL)   \\
\none &    &  \\
\none[r_2] &    *(boxL) &   *(boxL) &  *(boxL)   \\
\none &  &    &
\end{ytableau} } =   \frac{ \begin{ytableau}
\none[r_1] &   *(boxL) \\
\none[r_2] &   *(boxL)  &  *(boxL) 
\end{ytableau}\,\,\begin{ytableau}
\none[r_1] &  *(boxL) \\ \none[r_2] &    *(boxL) 
\end{ytableau} }{ \begin{ytableau}
\none[r_1] &   *(boxL) &  *(boxL)  \\
\none[r_2] &  *(boxL) &  *(boxL)  &  *(boxL) 
\end{ytableau} }\circ
\frac{ \begin{ytableau}
 \\
   &  
\end{ytableau}\,\,\begin{ytableau}
  \\
  &  
\end{ytableau} }{ \begin{ytableau}
\\
  &  \\
 &   & 
\end{ytableau} },
\end{equation}
for $(r_1,r_2) \in \BZ_{\geq 0}^2$.

\end{example}


Combining both windowing and factorization we arrive at the following 7-parameter families of $c=2$ triples of partitions, all which reduce to $\JackjLR_{21,21}^{321}$,

\ytableausetup{boxsize=1.1em}
\begin{equation}\label{eq:7paramwindowronly1}
 \windfact_{R_{12|3}} = \frac{\begin{ytableau}
\none[m_2\,]  &*(boxU)  & \none[1]  \\
\none   &*(boxS) a_1 &*(boxL) \\
\none[r_1]  &*(boxL)  & *(boxL) \\
\none[m_1\,]   &*(boxU)  & *(boxG)  & *(boxU) & \none[2] \\
\none    &*(boxS) a_2 &*(boxL)  &*(boxS)  a_3& *(boxL)   \\
\none[r_2] &*(boxL) & *(boxL)   &*(boxL)  & *(boxL)  \\
\none  &\none & \none[n_1]  & \none & \none[n_2]
\end{ytableau} \,\, \begin{ytableau}
\none \\
\none[m_3\,] & *(boxU) & \none[3] \\
\none & *(boxS) b_1  \\
\none[r_1] & *(boxL)  \\
\none[r_2] & *(boxL)  \\
\none &  *(boxS) b_2 & *(boxS) b_3  \\
\none & \none  & \none 
\end{ytableau}
}{\begin{ytableau}
\none & *(boxS) c_1\\
\none[m_3\,] & *(boxU) & \none[3] \\
\none[m_2\,] & *(boxU)  & \none[1]  \\
\none[r_1] & *(boxL) & *(boxL) & *(boxL) \\
\none & *(boxS) c_2 &*(boxL)  & *(boxS) c_3  \\
\none[m_1\,] & *(boxU) & *(boxG) &*(boxU) &\none[2]   \\
\none[r_2] & *(boxL)&*(boxL)  &*(boxL) & *(boxL)   & *(boxL)  \\
\none & *(boxS) c_4&*(boxL)  &*(boxS)  c_5 & *(boxL)  & *(boxS) c_6 \\
\none & \none  & \none[n_1] & \none  & \none[n_2] & \none
\end{ytableau}}\end{equation}

We also include the extra $r_2=0$ cases

\begin{equation}\label{eq:7paramwindowronly2}
 \windfact_{R_{12|34}} = \frac{\begin{ytableau}
\none[m_1\,]  &*(boxU)  & \none[1]  \\
\none   &*(boxS) a_1 &*(boxL) \\
\none[r_1]  &*(boxL)  & *(boxL) \\
\none[m_2\,]   &*(boxU)  & *(boxG)  & *(boxU) & \none[2] \\
\none    &*(boxS) a_2 &*(boxL)  &*(boxS)  a_3& *(boxL)   \\
\none  &\none & \none[n_1]  & \none & \none[n_2] 
\end{ytableau} \,\, \begin{ytableau}
\none \\
\none[m_3\,\,] & *(boxU) & \none[3] \\
\none & *(boxS) b_1  \\
\none[r_1] & *(boxL)  & \none & \none[4]  \\
\none &  *(boxS) b_2  & *(boxS) b_3 & *(boxL) \\
\none & \none   & \none & \none[n_4]
\end{ytableau}
}{\begin{ytableau}
\none & *(boxS) c_1\\
\none[m_3\,\,] & *(boxU)& \none[3] \\
\none[m_1\,] & *(boxU)  & \none[1]  \\
\none[r_1] & *(boxL) & *(boxL) & *(boxL) \\
\none & *(boxS) c_2 &*(boxL)  & *(boxS) c_3  \\
\none[m_2\,\,] & *(boxU) & *(boxG) &*(boxU) &\none[2]&\none[4]   \\
\none & *(boxS) c_4&*(boxL)  &*(boxS)  c_5 & *(boxL)  &  *(boxL)  & *(boxS) c_6 \\
\none & \none  & \none[n_1] & \none  & \none[n_2] & \none[n_4] & \none 
\end{ytableau}}
\end{equation}

We conclude this work with a extension of our main conjecture \eqref{conj:generalsolution}.

\begin{conjecture}
The Rule $\JackjLR_{21,21}^{321}$ given by \eqref{eq:generalsolution} solves the  families \eqref{eq:7paramwindowronly1} and \eqref{eq:7paramwindowronly2}.
\end{conjecture}


The extra equations provided by the requirement of solving these factorized families lead to further constraints on the solution \eqref{eq:generalsolution}. The existence of this solution is the strongest evidence we present for the validity of the windowing \eqref{conj:jackwindowing} and factorization \eqref{conj:factorization} conjectures. We will explore the properties of this solution and provide a direct construction of it (i.e. non-computational) in a follow up work \cite{Mickler:2025}.


\section{Acknowledgements}
The author would like to thank Per Alexandersson and Arun Ram for helpful conversations during the development of this work.

\bibliographystyle{hep}
\bibliography{/Users/ryanmickler/Dropbox/Kennebunk/Archive/MasterArchive}

\begin{thebibliography}{KTT09}

\bibitem[AM23]{Alexandersson:2023}
P.~Alexandersson and R.~Mickler, \textsl{ New cases of the Strong {S}tanley
  Conjecture},
\newblock (2023), {\href{https://arxiv.org/abs/2309.13870}{arXiv:2309.13870v3} [math.CO]}.

\bibitem[BG22]{Bravi:2021}
P.~Bravi and J.~Gandini, \textsl{ Some combinatorial properties of skew {J}ack
  symmetric functions},
\newblock The Electronic Journal of Combinatorics \textbf{ 29}(2) (2022),
 {\href{https://arxiv.org/abs/2107.00453}{arXiv:2107.00453v1} [math.CO]}.

\bibitem[CJ13]{CaiJing:2013}
T.~W. Cai and N.~Jing, \textsl{ {J}ack vertex operators and realization of
  {J}ack functions},
\newblock Journal of Algebraic Combinatorics \textbf{ 39}(1), 53--74 (mar
  2013).

\bibitem[KTT09]{King:2009}
R.~C. King, C.~Tollu and F.~Toumazet, \textsl{ Factorisation of
  Littlewood--Richardson coefficients},
\newblock Journal of Combinatorial Theory, Series A \textbf{ 116}(2), 314--333
  (2009).

\bibitem[Mic25]{Mickler:2025}
R.~Mickler, \textsl{ The structure of Jack Littlewood-Richardson coefficients},
\newblock In Preparation  (2025).

\bibitem[Sta89]{Stanley:1989}
R.~P. Stanley, \textsl{ Some combinatorial properties of {J}ack symmetric
  functions},
\newblock Advances in Mathematics \textbf{ 77}(1), 76--115 (1989).

\end{thebibliography}
\end{document}